\documentclass{article}

\usepackage{amsmath, amssymb, amsthm, mathrsfs}
\usepackage{indentfirst}
\usepackage{graphicx, tikz}
\usepackage{geometry}
\usepackage{caption}
\usepackage{subfigure, booktabs}

\usepackage{hyperref}

\usepackage{cleveref}
\crefname{section}{Section}{Sections}
\crefname{figure}{Figure}{Figures}
\crefname{table}{Table}{Tables}
\crefname{equation}{}{}
\crefname{theorem}{Theorem}{Theorems}
\crefname{lemma}{Lemma}{Lemmas}
\crefname{remark}{Remark}{Remarks}
\crefname{problem}{Problem}{Problems}

\newtheorem{theorem}{Theorem}[section]
\newtheorem{example}{Example}[section]
\newtheorem{problem}{Problem}[section]

\newtheorem{lemma}{Lemma}[section]

\graphicspath{{Figures/}}

\begin{document}
	
\title{The inverse obstacle scattering with incident tapered waves}

\author{
	Deyue Zhang\thanks{School of Mathematics, Jilin University, Changchun, China, {\it dyzhang@jlu.edu.cn}},
Mengjiao Bai\thanks{School of Mathematics, Jilin University, Changchun, China, {\it baimj24@mails.jlu.edu.cn}}, 
   	  Yan Chang\thanks{School of Mathematics, Harbin Institute of Technology, Harbin, China. {\it 21B312002@stu.hit.edu.cn}}
	     	 \ and 
	  Yukun Guo\thanks{School of Mathematics, Harbin Institute of Technology, Harbin, China. {\it ykguo@hit.edu.cn} (Corresponding author)}
}
\date{}

\maketitle

\begin{abstract}
	This paper is concerned with the reconstruction of the shape of an acoustic obstacle. Based on the use of the tapered waves with very narrow widths illuminating the obstacle, the boundary of the obstacle is reconstructed by a direct imaging algorithm. The stability of the imaging scheme is mathematically analyzed. We emphasize that different from the incident plane waves or point sources, the tapered waves with narrow widths bring several benefits in the inverse scattering: 1. local property. A tapered wave can illuminate only on a local part of the boundary of the obstacle, which generates the scattered field; 2. high resolution. We need only reconstruct the boundary near the beam, which improves the quality of some well-known algorithms; 3. fast and easy to implement. Numerical examples are included to demonstrate the effectiveness of the tapered waves.

\end{abstract}



\section{Introduction}

The inverse scattering problems are significant in diverse applications such as radar sensing, sonar detection, and biomedical imaging (see, e.g. \cite{Colton}). In the last three decades, numerous computational attempts have been made to solve the inverse scattering problems of identifying impenetrable obstacles or penetrable medium. Typical numerical strategies developed for the inverse scattering problems include decomposition methods, iteration schemes, recursive linearization-based algorithms, and sampling approaches (see, e.g. \cite{BLLT15, Cakoni1, CK18, Colton}). We refer to \cite{BL20, CDLZ20, DCL21} for some recent studies on the unique recovery issues in inverse scattering theory. In addition, inverse scattering problems without phase information have recently received great interest. Some uniqueness results and numerical methods on the exterior inverse scattering problems with phaseless data can be found in \cite{CH17, KR17, LL15, LLW17, ZZ17}.

Usually, in archetypal inverse scattering problems, the target objects are illuminated by an incident plane wave or an incident point source. It is well known that the range of irradiation of a plane wave or a point source is wide, and a large part of the boundary of the obstacle is illuminated directly. This means the scattered data includes all the information on the illuminated part of the boundary. Therefore, it is difficult to recover some local parts of the boundary accurately.
   
To overcome this difficulty, the main idea is to choose an incident field with a narrow width, which illuminates directly only on a local part of the boundary of the obstacle. This implies that the scattered fields mainly include information on this local part of the boundary, and the accuracy of the reconstruction of this part can be improved. To this aim, we introduce the Thorsos tapered wave \cite{Thorsos}, which is given by 
\begin{align}
u^i (x)&= \mathrm{exp}\left(\mathrm{i}k(x_1 \sin \theta_i-  x_2 \cos\theta_i)\right)(1+w(x))  \mathrm{exp}\left(-\frac{(x_1+x_2\tan \theta_i )^2}{g^2}\right)\notag\\ 
&\triangleq \mathrm{e}^{\mathrm{i}k x \cdot d}  (1+w(x)) \mathrm{e}^{-\frac{(x  \cdot d^\perp)^2}{\lambda^2}},   \label{tapered wave}
\end{align}
and satisfies the non-homogeneous Helmholtz equation
\begin{align}
\Delta u^i+ k^2 u^i= k^2 F  \quad \text{in}\ \mathbb{R}^2,\label{EqHelmholtz}
\end{align}
where 
\begin{equation}\label{Fx}
F=u^{i}(x)\left\{-w^2 (x)-16\frac{(x \cdot d)^2 (x  \cdot d^{\perp})^2} {k^4\lambda^8}  +\frac{4\mathrm{i}kx \cdot d}{k^4\lambda^4}\left(1-\frac{4(x \cdot d^\perp )^2}{\lambda^2}\right)\right\},
\end{equation}
and 
\begin{equation}\label{wx}
w (x)=\frac{1}{k^2\lambda^2}  \left(\frac{2(x  \cdot d^\perp )^2}{\lambda^2}-1\right).
\end{equation}
Here $k$ is the wave number, $x=(x_1,x_2)$, $d=(d_1, d_2)=(\sin \theta_i, -\cos \theta_i)$, $d^\perp=(-d_2,d_1)$, $d_1d_2\neq 0$, $\theta_i (0<|\theta_i|<\pi, |\theta_i|\neq \frac{\pi}{2})$ is the angle of incidence (it is the angle between the direction of propagation and the negative $x_2$ axis), $\lambda=g\cos\theta_i $ and $g$ is the parameter that controls the tapering. The tapered waves with $k\lambda\gg1$ are widely used for the inverse scattering by an unbounded rough surface \cite{Thorsos, Durnin87, ZWFL16, HZR08}.

In this paper, we consider the tapered wave incidence with $\lambda\ll1$, which, together with the definition \eqref{tapered wave},  means that the tapered wave has a very narrow width. As shown in \Cref{fig: taper}, the amplitude of the tapered wave decays exponentially away from the bandwidth. 

\begin{figure}
  \centering
	\subfigure[$g=0.5$]{\includegraphics[width=0.45\linewidth]{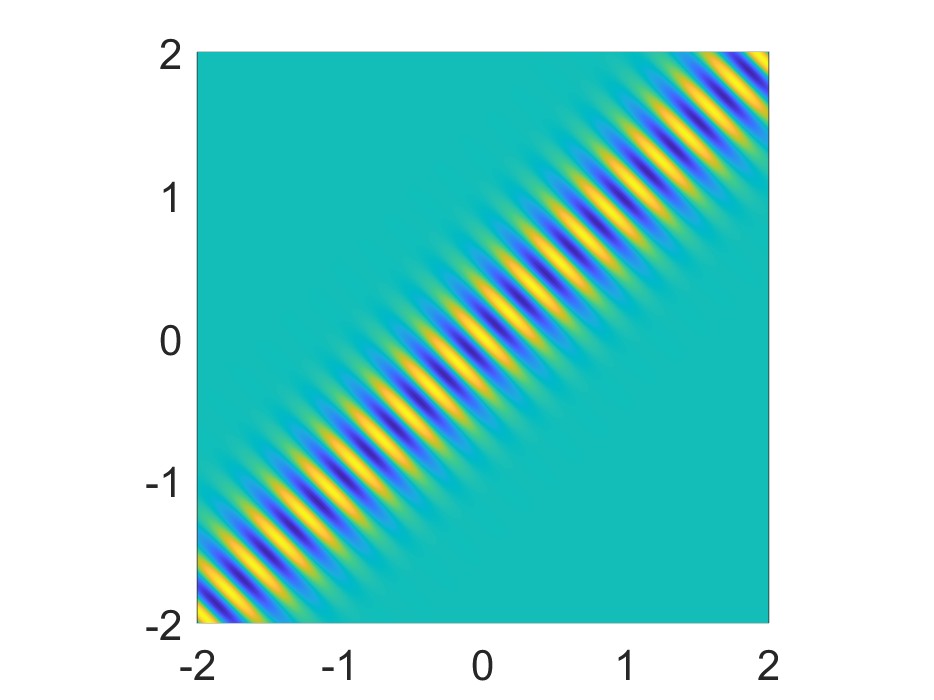}} \subfigure[$g=0.1$]{\includegraphics[width=0.45\linewidth]{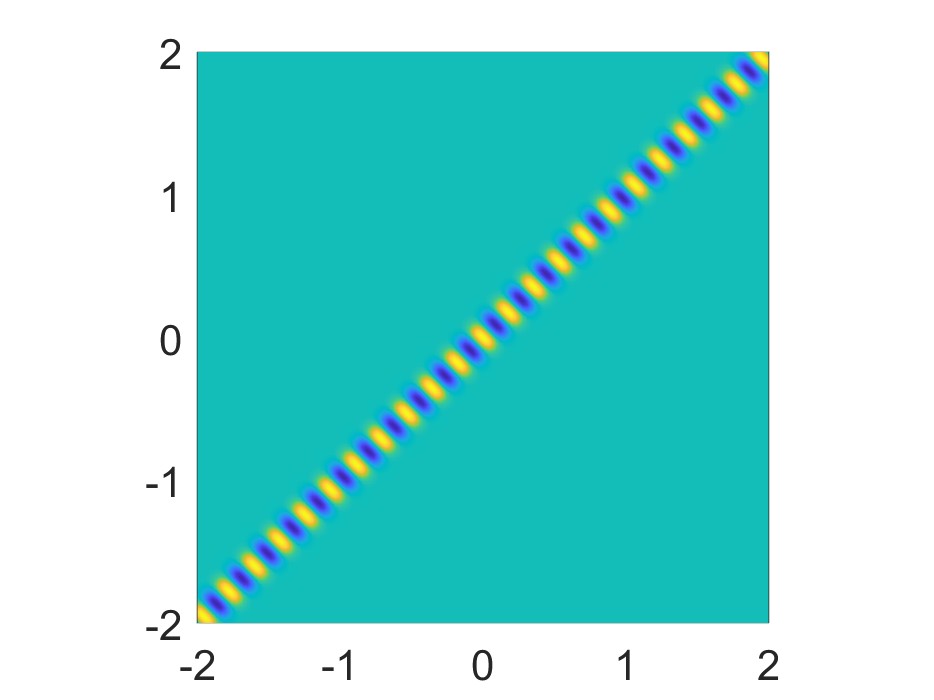}} 
  \caption{An illustration of the tapered wave $(k=25,\theta =\frac{3}{4}\pi)$.}\label{fig: taper}
\end{figure}

Based on this tapered wave incidence, we consider the inverse obstacle scattering problems. With a high wavenumber, only the boundary inside the bandwidth $\lambda$ of the beam is reconstructed by a direct imaging method. A mathematical analysis of the stability is presented to justify the theoretical foundations of the method. Our study is based on the high-frequency asymptotics, namely the Kirchhoff or the physical optics approximation. However, our numerical experiments show that the high-frequency requirement could be relaxed to a certain extent. 

The tapered waves with narrow widths have some advantages in solving inverse scattering problems. First, this illumination is concentrated in a local portion of the boundary of the scatterer, thus the induced scattered field mainly involves the information of this part of interest. Second, with a relatively high wavenumber, the confinement of the local illumination leads to the high resolution of inversion. Hence, we need only reconstruct the boundary inside the beam piecewisely, which significantly improves the quality of the existing well-known algorithms; Moreover, the combination of the tapered wave and the sampling scheme is fast and easy to implement. In particular, for a single tapered wave incidence, one only needs to reconstruct the corresponding boundary patch/points located inside the wave beam. Then all these recovered points naturally form the final reconstruction of the target shape of the obstacle.  

The rest of this paper is arranged as follows. In the next section, we introduce the inverse obstacle scattering problem and the imaging method. Then the mathematical analysis of the indicating properties is also presented. Next, numerical validations and discussions of the proposed method are illustrated in \cref{sec: example}. Finally, some concluding remarks are given in section \cref{sec: conclusion}.

\section{The inverse scattering and physical optics approximation}\label{sec: obstacle}

We begin this section with the mathematical formulations of the model scattering problem. Throughout this paper, we assume that $D\subset\mathbb{R}^2$ is an open and simply connected domain with $C^2$ boundary $\partial D$. 
The total field $u$ is the superposition of the incident field $u^i$ and the scattered field $u^s$, namely, $u=u^i+u^s$.  To characterize the physical properties of distinct scatterers, the boundary operator $\mathscr{B}$ is introduced by
\begin{align}\label{BC}
	\mathscr{B}u=
	\begin{cases}
		u, & \text{for a sound-soft obstacle},  \\
		\partial_\nu u+ \mathrm{i} k\mu u,  & \text{for an impedance obstacle},
	\end{cases}
\end{align}
where $\nu$ is the unit outward normal to $\partial D$ and $\mu$ is a real parameter. This boundary condition \cref{BC} covers the Dirichlet/sound-soft boundary condition, the Neumann/sound-hard boundary condition ($\mu=0$), and the impedance boundary condition ($\mu\neq 0$).

The obstacle scattering problem can be formulated as finding the scattered field $u^s\in H^1_{\rm loc}(\mathbb{R}^2\backslash\overline{D})$ satisfying the following boundary value problem:
\begin{align}
    \Delta u^s+ k^2 u^s= &\, 0\quad \text{in}\ \mathbb{R}^2\backslash\overline{D},\label{Oeq:Helmholtz} \\
	\mathscr{B}(u^i+u^s)= &\, 0 \quad \text{on}\ \partial D, \label{OBC} \\
	\frac{\partial u^s}{\partial r}-\mathrm{i} ku^s= &\, o\left(r^{-\frac{1}{2}}\right), \quad r=|x| \to \infty, \label{SRC}
\end{align}
where the Sommerfeld radiation condition \cref{SRC} holds uniformly in all directions $x/|x|$. The existence of a solution to the direct scattering problem \cref{Oeq:Helmholtz,OBC,SRC} is well known (see, e.g., \cite{Colton}).

With these preparations, the inverse obstacle scattering problem can be stated as follows.
\begin{problem}[Inverse obstacle scattering]\label{prob: obstacle}
	Let $D\subset B_R=\{x\in \mathbb{R}^2: |x|<R \}$ be the impenetrable obstacle with boundary condition $\mathscr{B}$ and $\mathbb{S}^1=\{x\in\mathbb{R}^2:|x|=1 \}$. Given the near-field data
	$$
		\{u^s(x; d):  x\in\partial B_R,\ d\in \mathbb{S}^1,\ d_1d_2\neq 0\},
	$$
	for a fixed wave number $k$, determine the boundary $\partial D$ of the obstacle.
\end{problem}

We refer to \cref{fig:model_obstacle} for an illustration of the geometry setting of Problem \ref{prob: obstacle}. 
\begin{figure}[htp]
	\centering
	\begin{tikzpicture}[thick]
		\pgfmathsetseed{2021}
		\draw plot [smooth cycle, samples=8, domain={1:8}, xshift=1cm, yshift=0.5cm] (\x*360/8+5*rnd:0.8cm+1cm*rnd) [fill=lightgray] node at (1.2, 0.1) {$D$};
		\draw  (1, 0.5) circle (2.5cm) node at (-0.3, 3) {$\partial B_R$}; 
		\draw [->] (-0.1, 1.1)--(-0.8, 1.7) node at (-0.6, 0.8) {$u^s$};
		\draw (0, 1.5) arc(90:180:0.5cm);
		\draw (0, 1.7) arc(90:180:0.7cm);
	    \draw [->] (1.88, 4.1)--(1.75, 3.3);
		\draw [->] (1.58, 4.1)--(1.45, 3.3) node at (2.1, 3.6) {$u^i$};
	\end{tikzpicture}
	\caption{An illustration of the inverse obstacle scattering problem.} \label{fig:model_obstacle}
\end{figure}
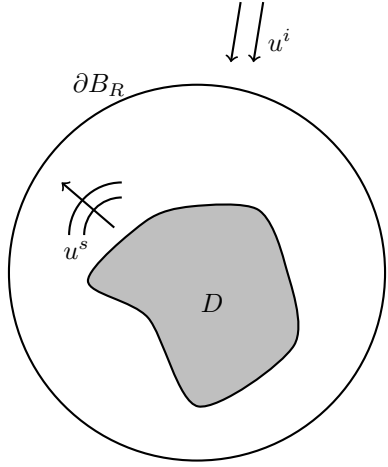

\subsection{Physical optics approximation and the indicator functions}

For a high wavenumber $k$ and a small $\lambda$, from the physical optics approximation, we assume that the tapered wave illuminates $\Gamma_d$ which is a small part of the boundary $\partial D$. Specifically, the measure of $\Gamma_d$ is represented by $|\Gamma_d|=C_0\lambda$ for some constant $C_0>1$. Let  
\begin{equation*}
\Phi(x,y)=\frac{\rm i}{4}H_0^{(1)}(k|x-y|), \quad x\neq y,
\end{equation*}
where $H_0^{(1)}$ denotes the zeroth-order Hankel function of the first kind. $\Phi$ is the fundamental solution to the Helmholtz equation. As a straightforward consequence of the Green formula in Theorem 2.4 in \cite{Colton}, we have the following lemma.
\begin{lemma}\label{Lemma2.1}
 For the scattering of a tapered wave incidence $u^i$ from an obstacle $D$, we have  
	\begin{equation}\label{Fomula1}
	u^s(x; d)=\int_{\Gamma_d}\left\{u^s(y;d)\frac{\partial \Phi(x,y)}{\partial \nu(y)}-\frac{\partial u^s}{\partial \nu}(y;d)\Phi(x,y)\right\}\mathrm{d}s(y),
    \quad x\in \mathbb{R}^2\backslash\overline{D}.	
\end{equation}
\end{lemma}	
	
From \cref{Lemma2.1}, $|\Gamma_d|=C_0\lambda$, $\lambda\ll 1$ and the numerical integration, we find 
\begin{equation}\label{Express}
	u^s(x; d)=\left(C_1\Phi(x, y_d)+C_2\partial_{\nu_y} \Phi(x, y_d)\right)\lambda+\mathcal{O}\left(\lambda^2\right), \quad \forall x\in \partial B_R,
\end{equation}
where $y_d\in \Gamma_d$ or $y_d\in \mathbb{R}^2\backslash\overline{D}$ with $\sup\limits_{x\in \Gamma_d}|y_d-x|\ll1$, and $C_1, C_2$ are some constants. 

From \eqref{Express}, we only need to find $y_d$ to determine the location of the boundary $\Gamma_d$. First, the following uniqueness result holds. 
\begin{theorem}\label{Thm2.1}
	Let $y_1, y_2 \in B_R\backslash D$, and $a_1, a_2, b_1, b_2$ be constants. Assume that
	\begin{equation}\label{unique}
	a_1\Phi(x, y_1)+b_1\partial_{\nu_y} \Phi(x, y_1)=  a_2\Phi(x, y_2)+b_2\partial_{\nu_y} \Phi(x, y_2), \quad \forall x\in \partial B_R.
	\end{equation}
	Then we have $y_1=y_2$, $a_1=a_2$ and $b_1=b_2$.
\end{theorem}
\begin{proof}
First, we claim that $y_1=y_2$. Otherwise, if $y_1\neq y_2$, let
\begin{equation*} 
	v_1(x)=a_1\Phi(x,y_1)+b_1\partial_{\nu_y} \Phi(x,y_1),\quad   	v_2(x)=a_2\Phi(x,y_2)+b_2\partial_{\nu_y} \Phi(x,y_2), \quad x\notin \{y_1, y_2\}.
\end{equation*}	
Then from \eqref{unique}, we have $v_1=v_2$ on $\partial B_R$, which, together with the uniqueness of the scattering problem in $\mathbb{R}^2\backslash \overline{B_R}$, yields $v_1=v_2$ in $\mathbb{R}^2\backslash \overline{B_R}$. Further, by using the analyticity of
 $v_1$ and $v_2$, we deduce 	
\begin{equation*} 
	v_1(x)=v_2(x), \quad \forall x\in \mathbb{R}^2\backslash\{y_1, y_2\} .
\end{equation*}		
By letting $x\rightarrow y_2$ and using the boundedness of $v_1(x)$, we have that $v_1(x)$ is bounded and $v_2(x)$ tends to infinity, which is a contradiction. Hence $y_1=y_2$.	
	
Since $y_1=y_2$, let
	\begin{equation*} 
		w_1(x)=(a_1-a_2)\Phi(x,y_1),\quad   	w_2(x)=(b_2-b_1)\partial_{\nu_y} \Phi(x,y_1), \quad x\neq y_1.
	\end{equation*}	
Then from \eqref{unique}, we have $w_1=w_2$ on $\partial B_R$. Again, by using the uniqueness of the scattering problem in $\mathbb{R}^2\backslash \overline{B_R}$, we have $w_1=w_2$  in $\mathbb{R}^2\backslash \overline{B_R}$, and thus 
\begin{equation*} 
	w_1^\infty(\hat{x})=w_2^\infty(\hat{x}),\quad\forall \hat{x}\in \mathbb{S}^1,
\end{equation*}	
where $w_1^\infty$ and $w_2^\infty$ are the far-field patterns of $w_1$ and $w_2$, respectively.
Further, from the  far-field patterns of $\Phi(x,y_1)$ and $\partial_{\nu_y} \Phi(x, y_1)$, it can be seen that
	\begin{equation*} 
	(a_1-a_2)\frac{{\rm e}^{{\rm i}\frac{\pi}{4}}}{\sqrt{8\pi k}} 
	{\rm e}^{-{\rm i}k \hat{x}\cdot y_1}=-{\rm i}k(b_2-b_1)\nu(y_1)\cdot \hat{x}\frac{{\rm e}^{{\rm i}\frac{\pi}{4}}}{\sqrt{8\pi k}} 
	{\rm e}^{-{\rm i}k \hat{x}\cdot y_1}, \quad\forall \hat{x}\in \mathbb{S}^1,
\end{equation*}	
which implies
\begin{equation*} 
a_1-a_2={\rm i}k(b_1-b_2)\nu(y_1)\cdot \hat{x}, \quad\forall \hat{x}\in \mathbb{S}^1.
\end{equation*}	
This means $a_1=a_2$ and $b_1=b_2$. The proof is completed.
\end{proof}

Consider the measured noisy data $u^s_\delta( \cdot; d)\in L^2(\partial B_R)$ satisfying $\|u^s( \cdot; d)-u^s_\delta( \cdot; d)\|_{ L^2(\partial B_R)}\leq \delta \|u^s( \cdot; d)\|_{ L^2(\partial B_R)}$ with $0<\delta<1$.

To determine $y_d$, we introduce the following indicator function
\begin{equation}\label{Indicator}
I(z; d)= \left|\frac{1}{\lambda}\left\langle u^s_\delta( \cdot; d), {\rm e}^{{\rm i}\left(k|\cdot-z|+\frac{\pi}{4}\right)}\right\rangle_{L^2(\partial B_R)}\right|, 
\end{equation}
where $\langle\cdot,\cdot \rangle$ denotes the $L^2$-inner product. We
take the first $M$ maximum points of the indicator function $I(z;d)$ as the reconstruction of the boundary $\Gamma_d$.

The following indicating properties hold for $I(z; d)$.
\begin{theorem}\label{Thm2.2}
	For $I(z;d)$, we have
\begin{align} \nonumber
I(z; d)&= \left| \frac{C_1}{4}\int_{\partial B_R} \sqrt{\frac{2}{\pi k|x-y_d|}}
{\rm e}^{{\rm i}k\left(|x-y_d|-|x-z|\right)}\mathrm{d}s(x)\right.
\\ \label{property}
&\quad \left. +\frac{C_2{\rm i }k}{4}\int_{\partial B_R} \sqrt{\frac{2}{\pi k|x-y_d|}}\partial_{\nu_y}|x-y_d|
{\rm e}^{{\rm i}\left(k|x-y_d|-k|x-z|\right)}\mathrm{d}s(x)\right|+\mathcal{O}\left(\delta^2 +\lambda+R^{-\frac{1}{2}}\right).
\end{align}
\end{theorem}
\begin{proof}
We first recall the following asymptotic behavior of the Hankel functions \cite[(3.105)]{Colton}
\begin{equation*} 
H_n^{(1)}(t) =  \sqrt{\frac{2}{\pi t}}{\rm e}^{{\rm i}\left(t-\frac{n\pi}{2}-\frac{\pi}{4}\right)}\left\{ 1+\mathcal{O}\left(\frac{1}{t}\right)\right\}, \quad t \rightarrow \infty.
\end{equation*}
Then, we have
\begin{align*} 
\left\langle \Phi(x,y_d), {\rm e}^{{\rm i}\left(k|\cdot-z|+\frac{\pi}{4}\right)}\right\rangle_{L^2(\partial B_R)}
&= \int_{\partial B_R} \frac{\rm i}{4}H_0^{(1)}(k|x-y_d|)
{\rm e}^{-{\rm i}\left(k|x-z|+\frac{\pi}{4}\right)} \mathrm{d}s(x) \\
&=  \frac{\rm 1}{4}\int_{\partial B_R} \sqrt{\frac{2}{\pi k|x-y_d|}}
{\rm e}^{{\rm i}k\left(|x-y_d|-|x-z|\right)}\mathrm{d}s(x)
+\mathcal{O}\left(R^{-\frac{1}{2}}\right),
\end{align*}
and
\begin{align*} 
&\left\langle \partial_{\nu} \Phi(x,y_d), {\rm e}^{{\rm i}\left(k|\cdot-z|+\frac{\pi}{4}\right)}\right\rangle_{L^2(\partial B_R)}
\\
=& -\int_{\partial B_R} \frac{{\rm i }k}{4}H_1^{(1)}(k|x-y_d|)
\partial_{\nu_y}|x-y_d|
{\rm e}^{-{\rm i}\left(k|x-z|+\frac{\pi}{4}\right)} \mathrm{d}s(x)
\\
=& \frac{{\rm i }k}{4}\int_{\partial B_R} \sqrt{\frac{2}{\pi k|x-y_d|}}\partial_{\nu_y}|x-y_d|
{\rm e}^{{\rm i}\left(k|x-y_d|-k|x-z|\right)}\mathrm{d}s(x)
+\mathcal{O}\left(R^{-\frac{1}{2}}\right).
\end{align*}
Further, from \eqref{Express},  we deduce that
\begin{align*} 
	I(z; d)&= \left|\frac{1}{\lambda}\left\langle u^s( \cdot; d), {\rm e}^{{\rm i}\left(k|\cdot-z|+\frac{\pi}{4}\right)}\right\rangle_{L^2(\partial B_R)}\right|+\left|\frac{1}{\lambda}\left\langle u^s( \cdot; d)-u^s_\delta( \cdot; d), {\rm e}^{{\rm i}\left(k|\cdot-z|+\frac{\pi}{4}\right)}\right\rangle_{L^2(\partial B_R)}\right| \\
	& = \frac{\delta}{\lambda}\|u^s( \cdot; d)\|_{L^2(\partial B_R)}+
	\left|\left\langle C_1\Phi(x, y_d)+C_2\partial_{\nu_y} \Phi(x, y_d),
	 {\rm e}^{{\rm i}\left(k|\cdot-z|+\frac{\pi}{4}\right)}\right\rangle_{L^2(\partial B_R)}\right|+\mathcal{O}(\lambda) \\
	&= \left|\left\langle 
	 C_1\Phi(x,y_d)+C_2\partial_{\nu_y} \Phi(x,y_d),
	 {\rm e}^{{\rm i}\left(k|\cdot-z|+\frac{\pi}{4}\right)}\right\rangle_{L^2(\partial B_R)}\right|+\mathcal{O}(\delta^2+\lambda)
	 \\
	 &= \left| \frac{C_1}{4}\int_{\partial B_R} \sqrt{\frac{2}{\pi k|x-y_d|}}
	 {\rm e}^{{\rm i}k\left(|x-y_d|-|x-z|\right)}\mathrm{d}s(x)\right.
	 \\
	 	&\quad \left. +\frac{C_2{\rm i }k}{4}\int_{\partial B_R} \sqrt{\frac{2}{\pi k|x-y_d|}}\partial_{\nu_y}|x-y_d|
	 	 {\rm e}^{{\rm i}\left(k|x-y_d|-k|x-z|\right)}\mathrm{d}s(x)\right|+\mathcal{O}\left(\delta^2 +\lambda+R^{-\frac{1}{2}}\right),
\end{align*}
which completes the proof.
\end{proof}

In virtue of the above theorem, we see that the indicator function $I(z; d)$ should decay as the sampling point $z$ recedes from the point $y_d$. In particular, for small $\delta, \lambda$ and large $R$, function $I(z;d)$ takes the local maximum value at $z=y_d$
\begin{align*} 
I(z;d)
\approx  \left| \frac{C_1}{4}\int_{\partial B_R} \sqrt{\frac{2}{\pi k|x-y_d|}}
\mathrm{d}s(x) +\frac{C_2{\rm i }k}{4}\int_{\partial B_R} \sqrt{\frac{2}{\pi k|x-y_d|}}\partial_{\nu_y}|x-y_d|
\mathrm{d}s(x)\right|.
\end{align*}
This indicating behavior will be numerically tested by the experiments in the next section.

We end this section with a brief description of our algorithm for the direct imaging scheme for reconstructing the shape of the scatterer.
\begin{table}[htp]
	\centering
	\begin{tabular}{cp{.8\textwidth}}
		\toprule
		\multicolumn{2}{l}{{\bf Algorithm:}\quad Reconstruction by the tapered wave incidence and the direct imaging scheme} \\
		\midrule
		{\bf Step 1} & Given a frequency $k$ and an incident direction $d\in  \mathbb{S}^1$, collect the corresponding noisy near-field data due to the tapered wave incidence and the obstacle. \\
		{\bf Step 2} & Select a suitable elongated imaging mesh  $\mathcal{T}$ covering the tapered wave incidence and some local part of the scatterer. For each imaging point $z\in \mathcal{T}$, evaluate the indicator function $I(z;d)$.\\
		{\bf Step 3} & The local boundary of the scatterer in $\mathcal{T}$ can be recovered as the first $M$ maximum points of the indicator function $I(z;d)$. \\
		{\bf Step 4} & Choose a new incident direction $d\in  \mathbb{S}^1$ and repeat steps 1-3, until all the finite directions are used up. \\
		{\bf Step 5} & Plot the locations of all the above maximum points. This output is the numerical reconstruction of the boundary of the scatterer. \\
		\bottomrule
	\end{tabular}
\end{table}

\section{Numerical examples}\label{sec: example}
In this section, we shall provide several numerical examples to illustrate the performance of the proposed method.

We use the boundary integral equation method \cite{CK18} to obtain the synthetic data for the inverse problem. Random noise is then added to the measured data to test the method's stability. The  noisy data is given by 
$$
u^{s,\delta}=u^s+\delta r_1|u^s|\mathrm{e}^{\mathrm{i}\pi r_2},
$$
where $r_1,r_2$ are two uniformly distributed random numbers ranging from $-1$ to 1, and $\delta>0$ is the noise level (without otherwise specified, $\delta=5\%$ is used).
In the following numerical examples, we take the incident directions to be $d_j=(\cos2j\pi/N_d,\sin2j\pi/N_d),\,j=1,2,\cdots, N_d$ with $N_d\in\mathbb{N}_+$ being the number of incident directions. To record the scattered data on the measurement curve, the receivers are chosen to be $x_i=5(\cos2i\pi/N_R,\sin2i\pi/N_R),$ and $N_R\in\mathbb{N}_+$ is the number of the measurement points. The sampling domain is chosen as $\Omega=[-2,2]\times[-2,2]$ with $150\times 150$ equally distributed sampling points.
 
It deserves noting that, our method produces a special imaging result due to the incident tapered waves, which will be illustrated in the later examples. As analyzed in the previous sections, for each $d$, the function $I(z; d)$ takes the local maximum value at $z=y_d\in\Gamma_d$ or $y_d\in \mathbb{R}^2\backslash\overline{D}$ with $\sup\limits_{x\in \Gamma_d}|y_d-x|\ll1$.
This property motivates us to take the first $M$ local maximizers of the indicator function $I(z;d_j)$ for each incident direction $d_j,\,j=1,\cdots N_d,$ to characterize the boundary of the obstacle. We mark these local maximizers with blue points, meanwhile, the exact boundary is plotted by the red curve for comparison. As an initial attempt to test the method, we only consider the sound-soft obstacles in this article.

\begin{example}
  In the first example, we consider the reconstruction of an obstacle of a circle shape, whose boundary can be parameterized by 
  $$
  x(t)=(\cos,\sin t),\quad t\in[0,2\pi).
  $$
  
  We first consider the reconstruction with $N_d=N_R=512$. By taking $k=25$, we exhibit the reconstruction in \Cref{fig: circle}. In \Cref{fig: circle}(b)--\Cref{fig: circle}(f), we show the reconstructions by taking $g=10^{-5}, 10^{-3}, 10^{-2}, 5\times10^{-2},$ and $2\times10^{-1}$, respectively. We can see from  \Cref{fig: circle}(b)--\Cref{fig: circle}(g) that the reconstructions match well with the exact scatterer when we utilize the incident tapered waves with small bandwidth. We want to point out that, the reconstructions in \Cref{fig: circle} are formulated pointwisely. We also observe that with a relatively large $g$ (for instance, $g=0.2$ in \Cref{fig: circle}(f)), the reconstruction may be less satisfactory. Therefore, a sufficiently small $g$ is usually indispensable for an accurate reconstruction.
  
  To further illustrate the intermediate procedure of the reconstruction, we list more details for the case $g=0.01$ in \Cref{fig: circle_illustrate}. We plot the incident field and the indicator imaging with $M=2$ corresponding to the $j=15,100,237$, and 430, respectively,  in \Cref{fig: circle_illustrate}. The incident fields in the sampling domains are depicted in the first column of \Cref{fig: circle_illustrate}. As one can easily see, the incident wave propagates in a certain direction with a narrow bandwidth. The second column shows the surface plots of the indicator function $I(z; d_j)$, $j=15,100,237,430$, respectively. Since $M=2$ is chosen, we can easily see two significant local peaks in each figure in the second column. In addition, we observe that the two local maximum points are located almost at the intersection points of the propagation direction and the scatterer. By locating the corresponding local maximum points in the third column, we mark these maximizers with blue points in the third column. As seen in the third column, the local extreme points are located exactly at the boundary of the obstacle. For each $j=1,\cdots, N_d,$ we plot the local extreme points pointwisely in \Cref{fig: circle}(b). When $N_d$ is relatively large ($N_d=512$ in this example), these separately distributed points could capture the boundary profile accurately. This illustrates that a highly accurate reconstruction is achieved by the proposed method.
Furthermore, we notice from \Cref{fig: circle}(f) that the value of $g$ is closely related to the reconstruction. If the parameter $g$ is not chosen properly, the proposed method may fail to produce a satisfactory recovery.
 
\end{example}

\begin{figure}
	\centering
	\subfigure[model setup]{\includegraphics[width=0.3\linewidth]{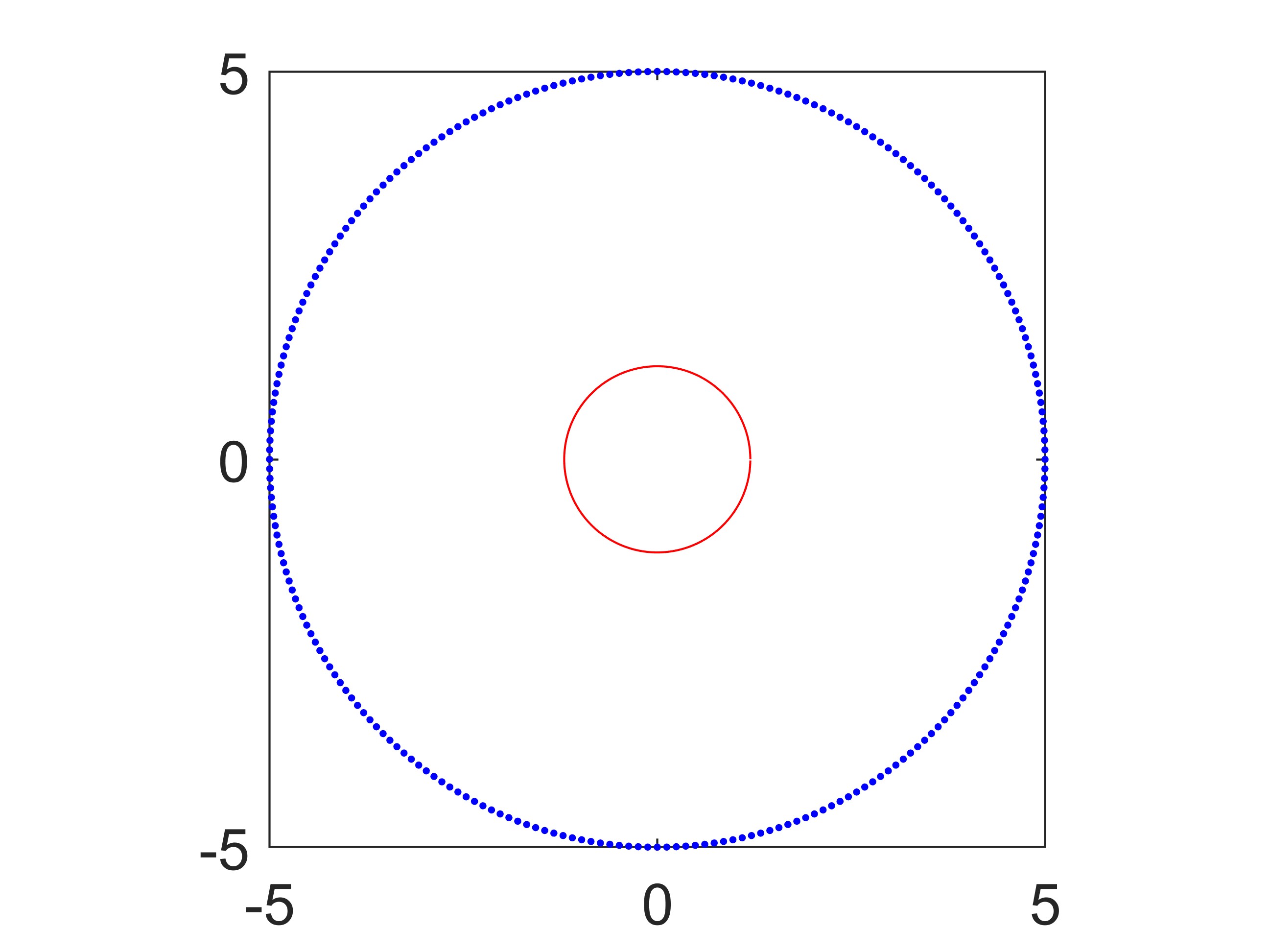}}
	\subfigure[$g=10^{-5}$]{\includegraphics[width=0.3\linewidth]{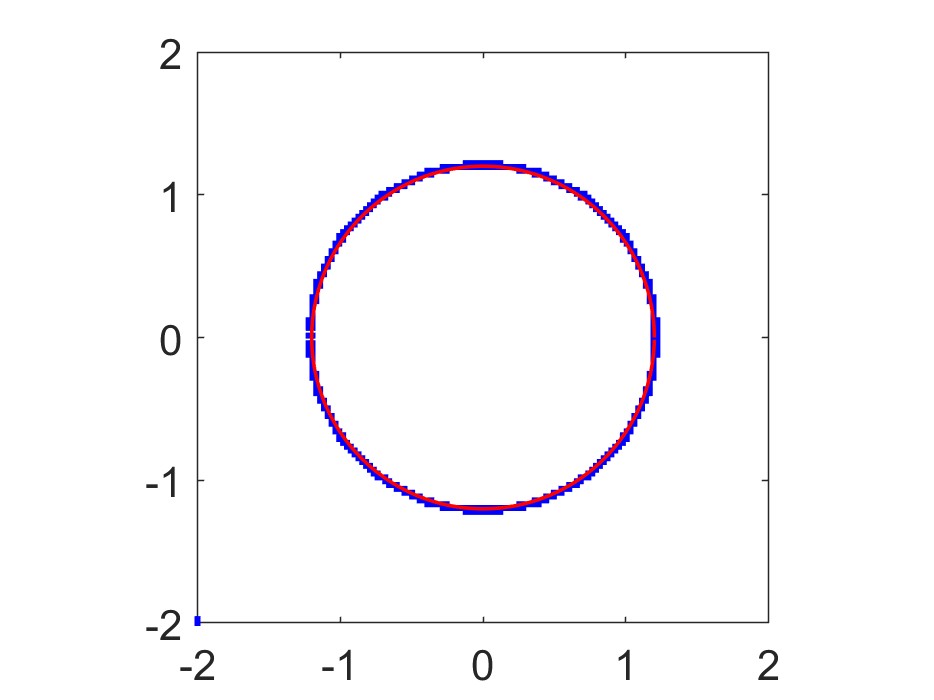}}
	\subfigure[$g=10^{-3}$]{\includegraphics[width=0.3\linewidth]{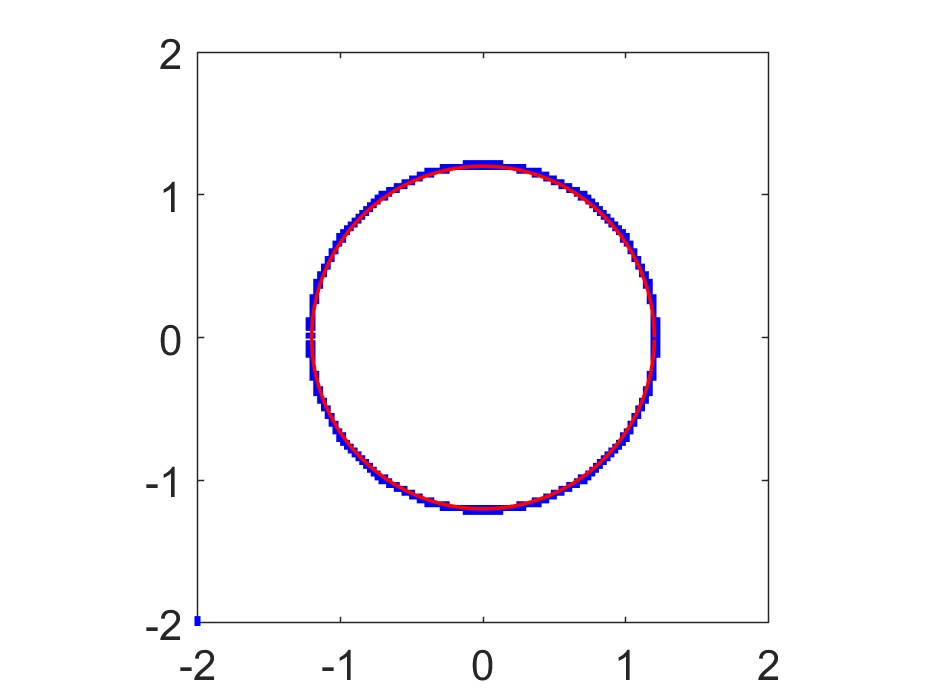}}
	\subfigure[$g=10^{-2}$]{\includegraphics[width=0.3\linewidth]{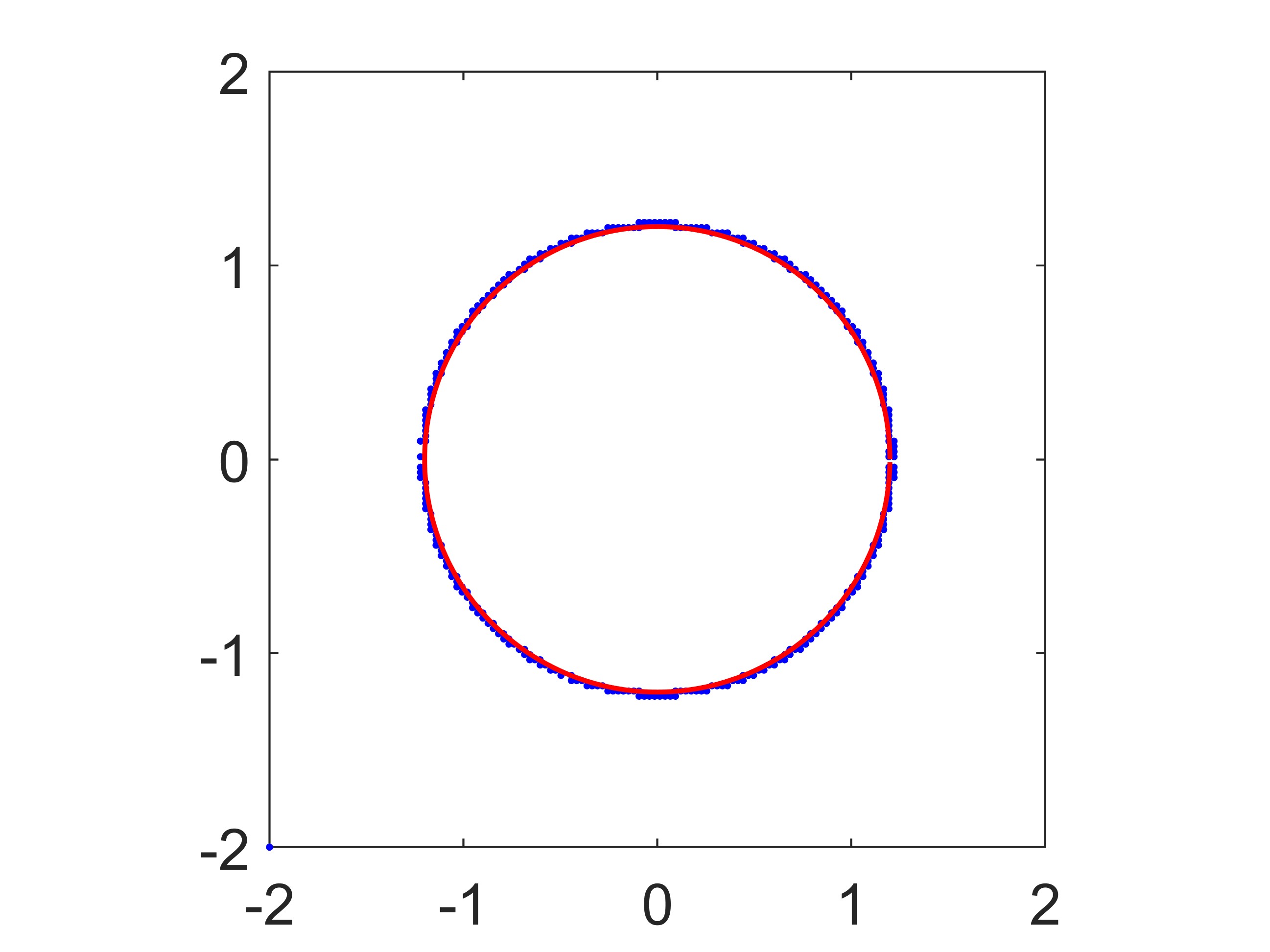}}
	\subfigure[$g=5\times10^{-2}$]{\includegraphics[width=0.3\linewidth]{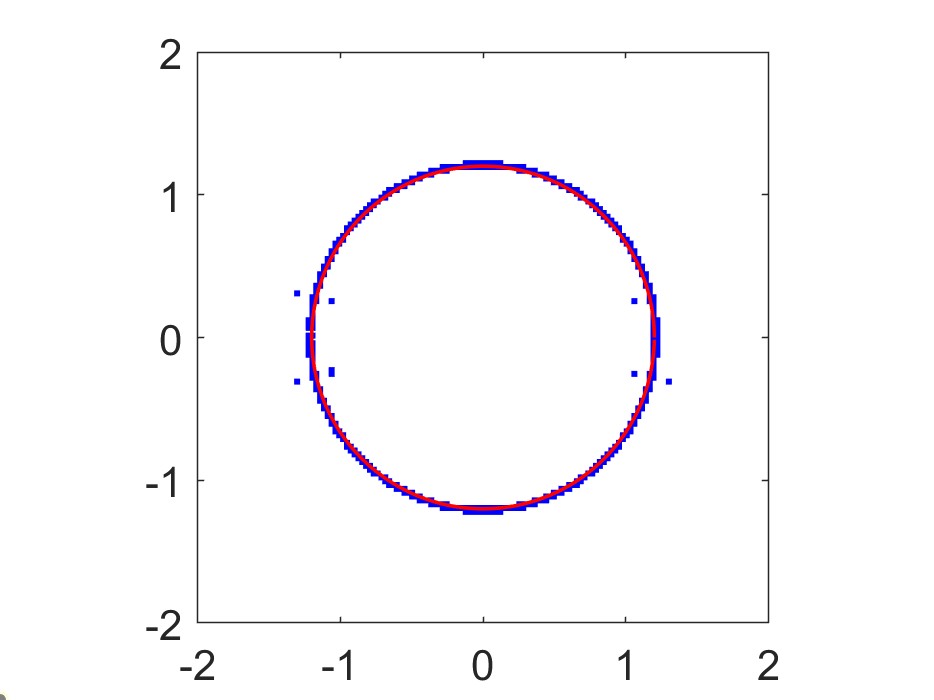}}
	\subfigure[$g=2\times10^{-1}$]{\includegraphics[width=0.3\linewidth]{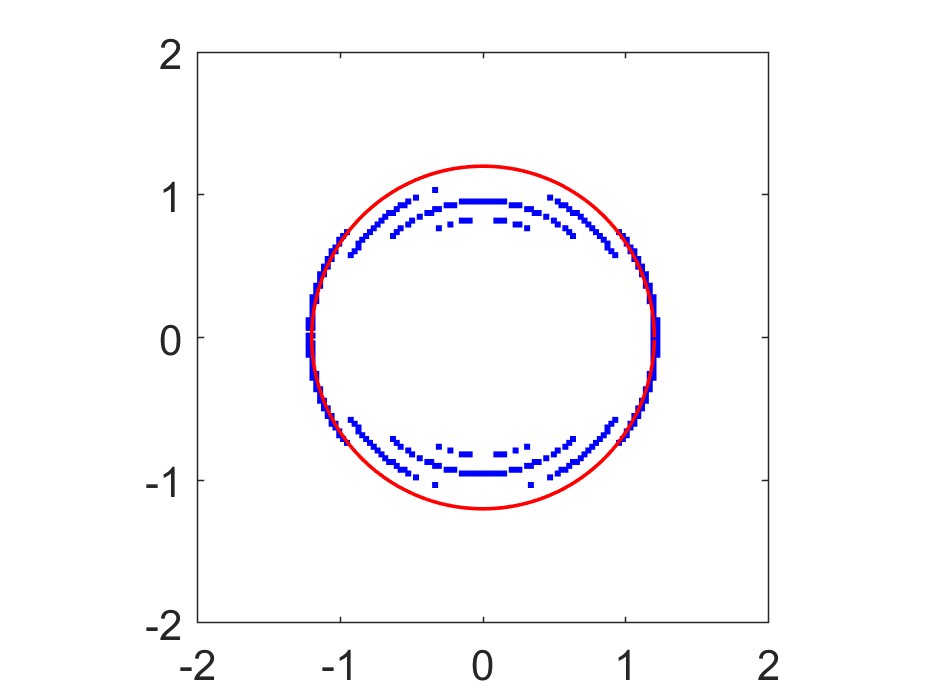}}
  \caption{Reconstruction of the circle-shaped obstacle with $k=25.$}\label{fig: circle}
\end{figure}

\begin{figure}
	\subfigure[]{\includegraphics[width=0.32\linewidth]{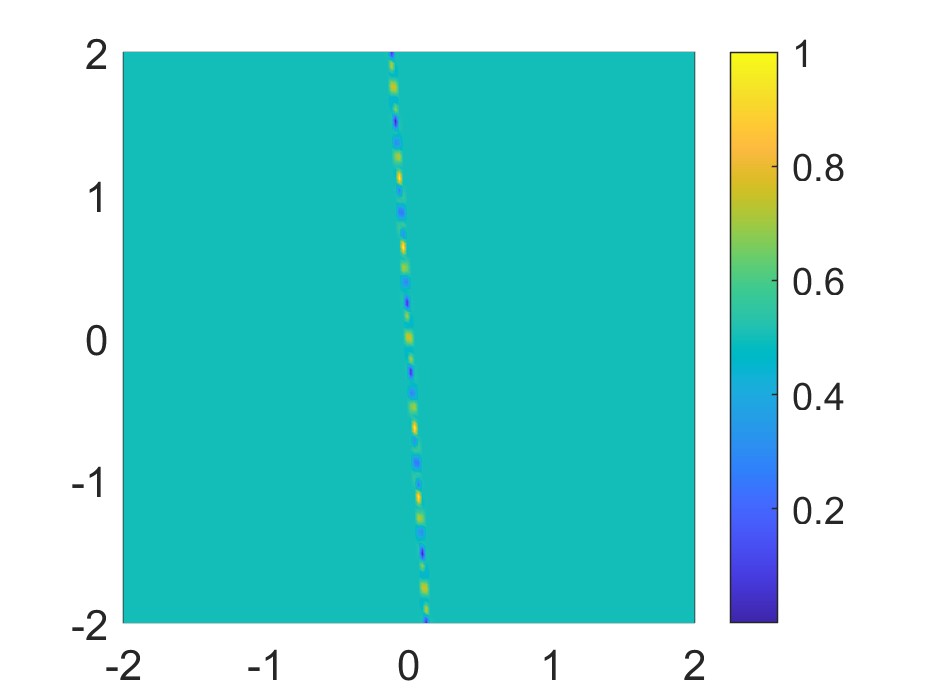}}
	\subfigure[]{\includegraphics[width=0.32\linewidth]{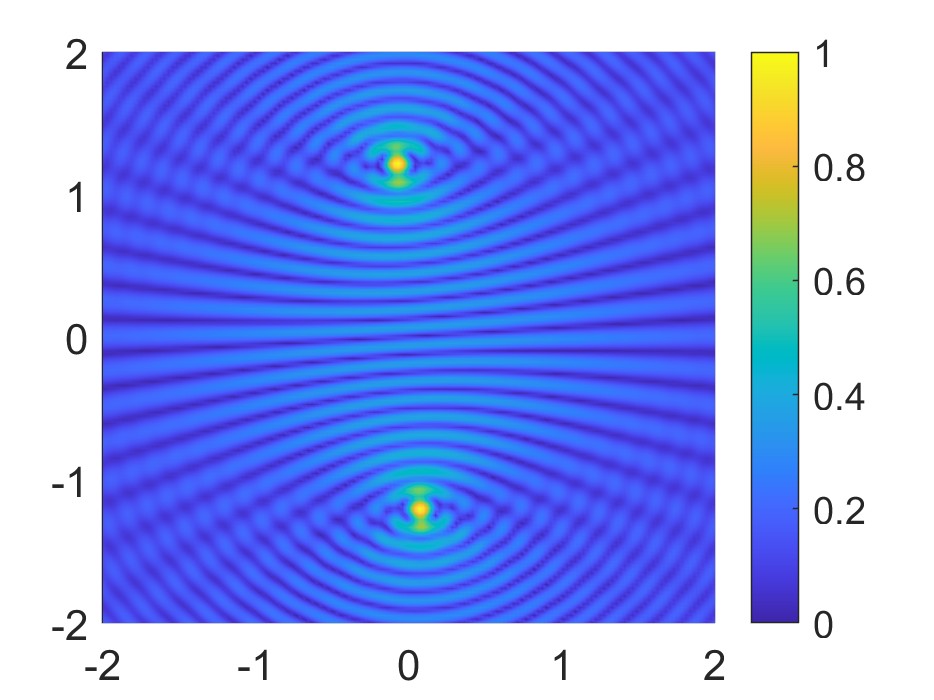}}
	\subfigure[]{\includegraphics[width=0.32\linewidth]{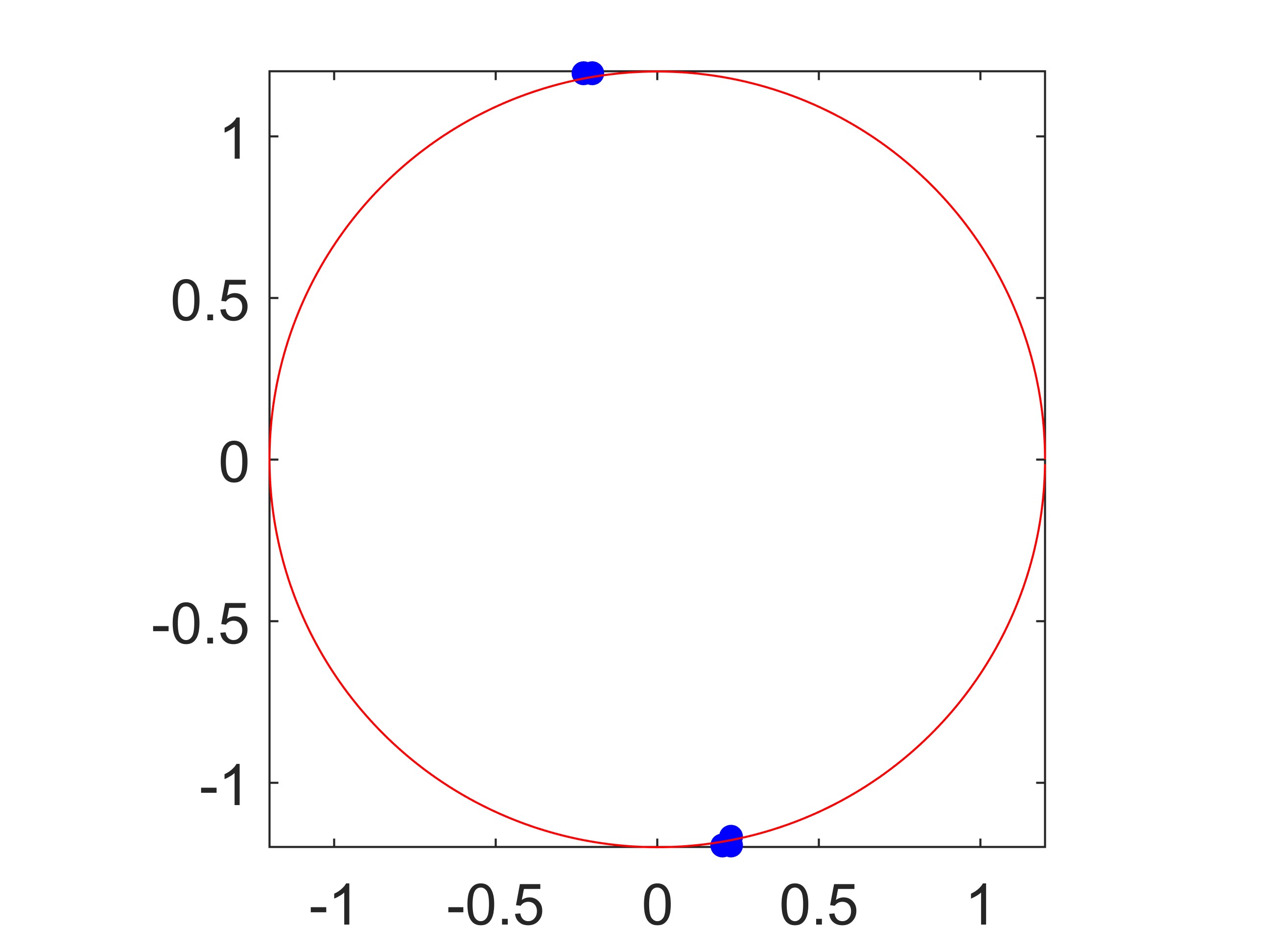}}\\
	\subfigure[]{\includegraphics[width=0.32\linewidth]{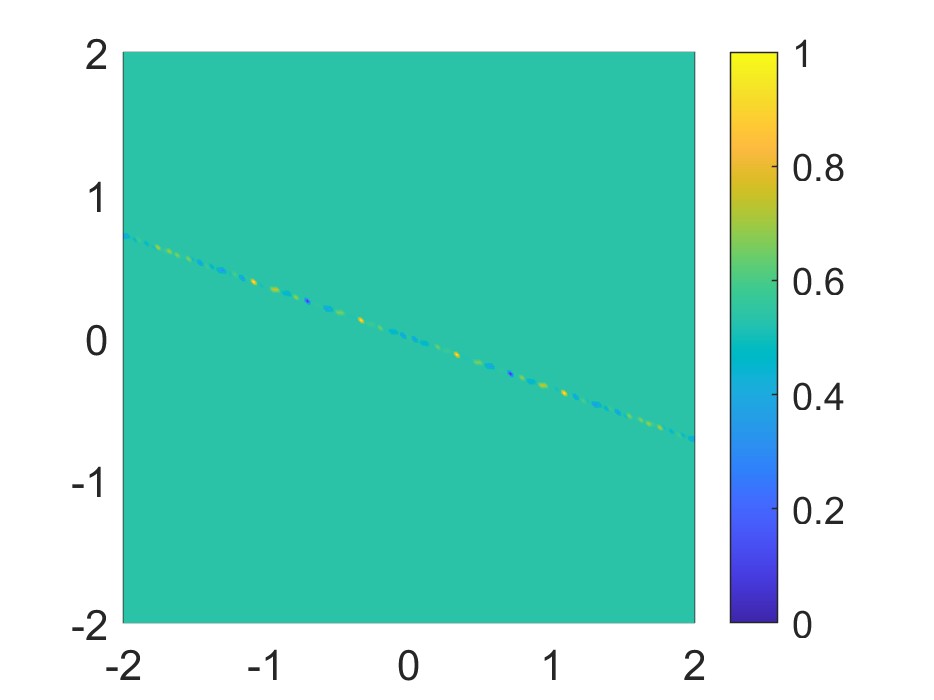}}
	\subfigure[]{\includegraphics[width=0.32\linewidth]{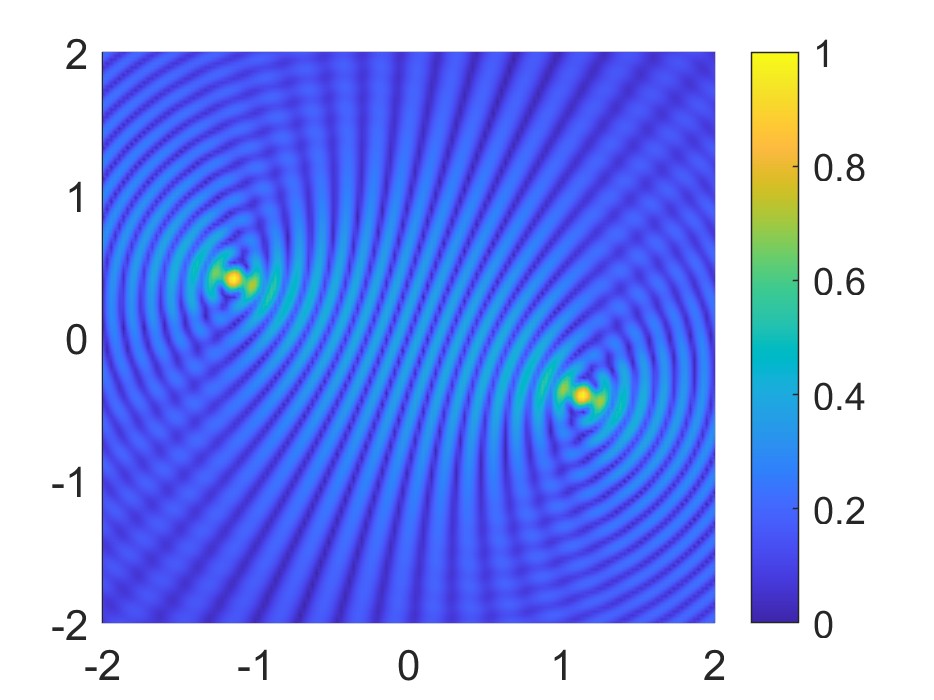}}
	\subfigure[]{\includegraphics[width=0.32\linewidth]{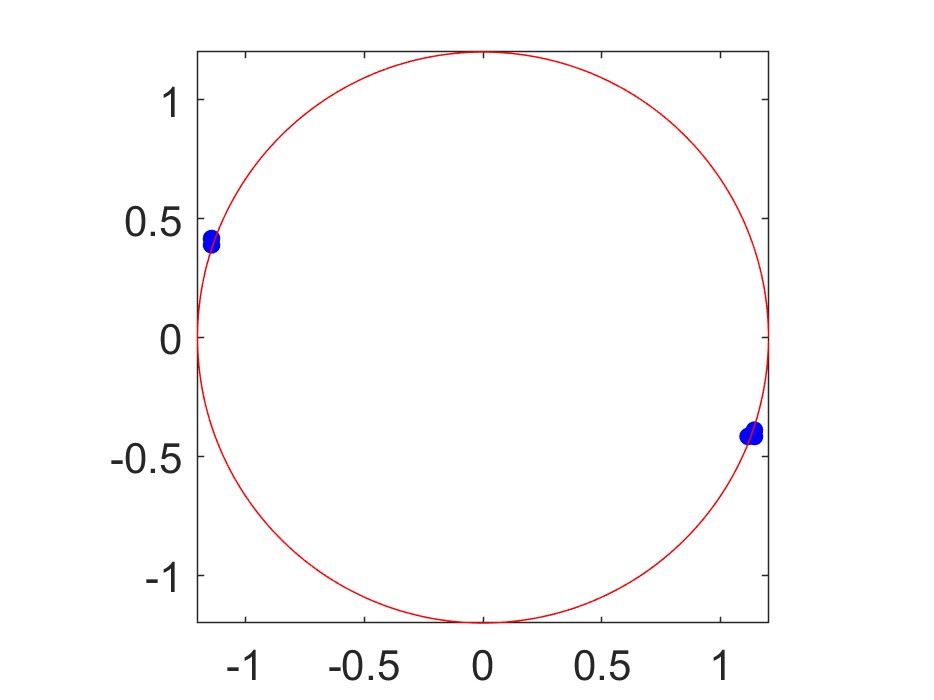}}\\
	\subfigure[]{\includegraphics[width=0.32\linewidth]{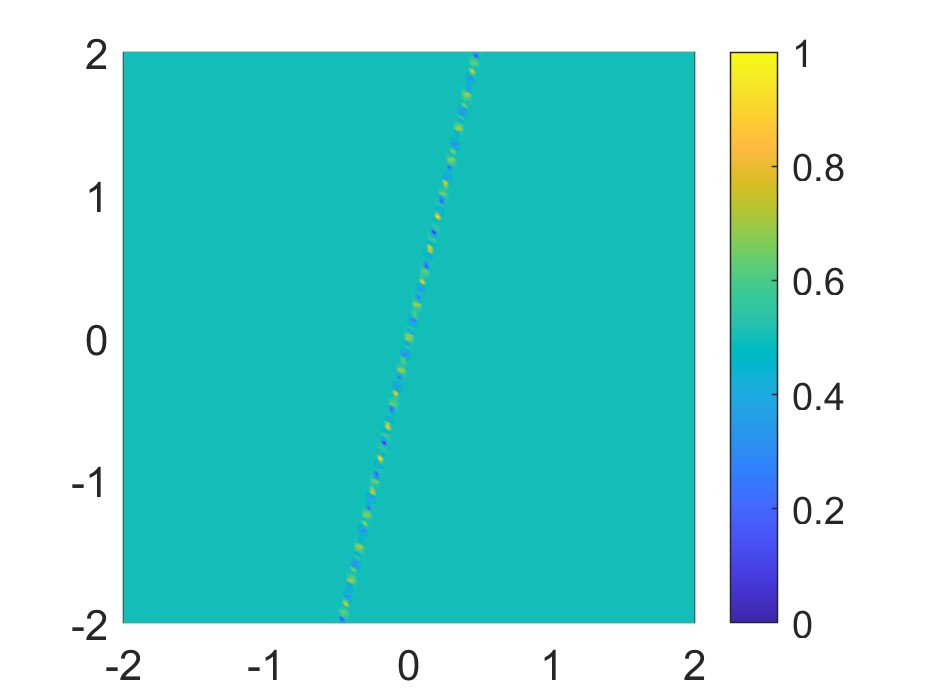}}
	\subfigure[]{\includegraphics[width=0.32\linewidth]{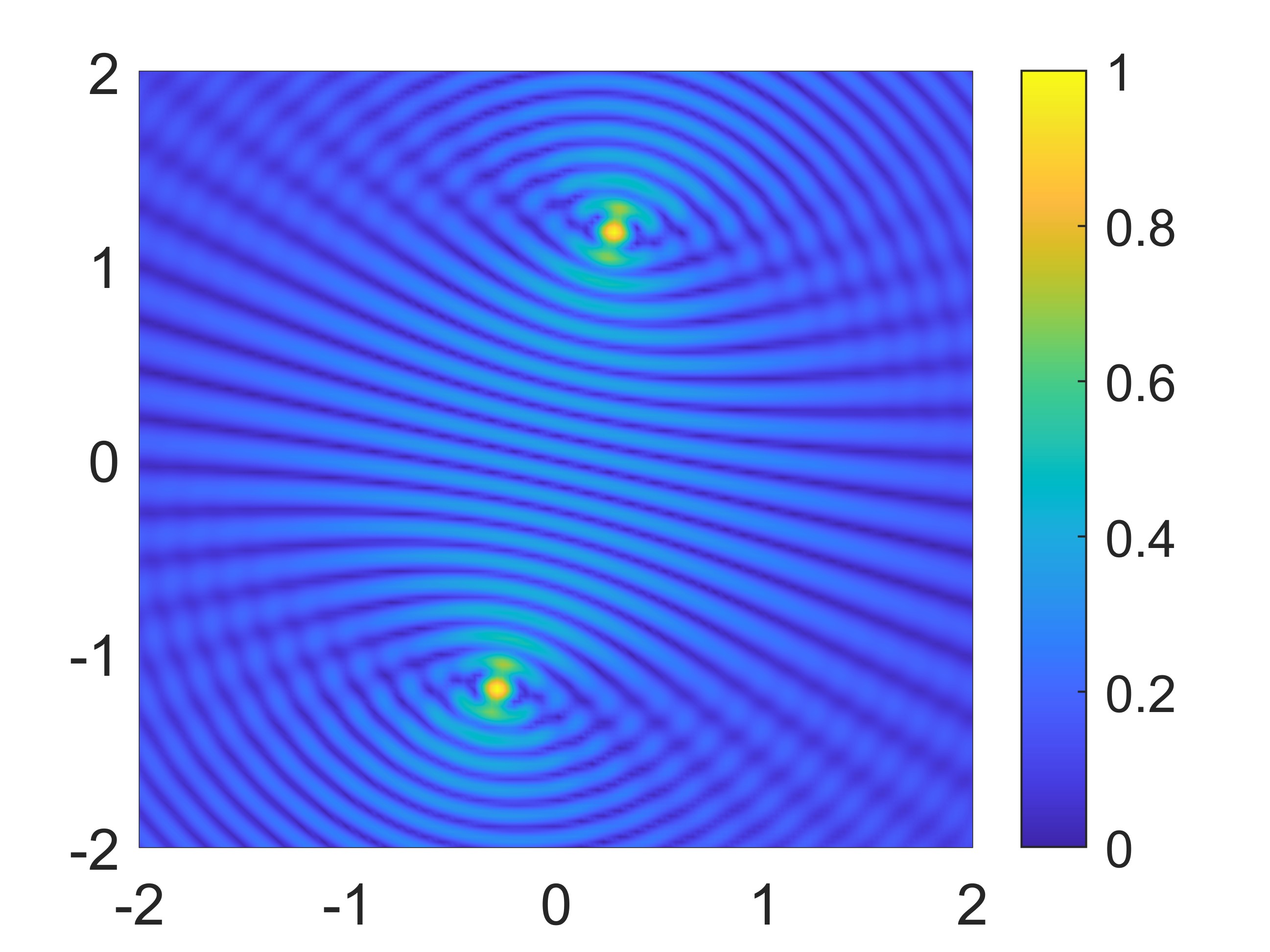}}
	\subfigure[]{\includegraphics[width=0.32\linewidth]{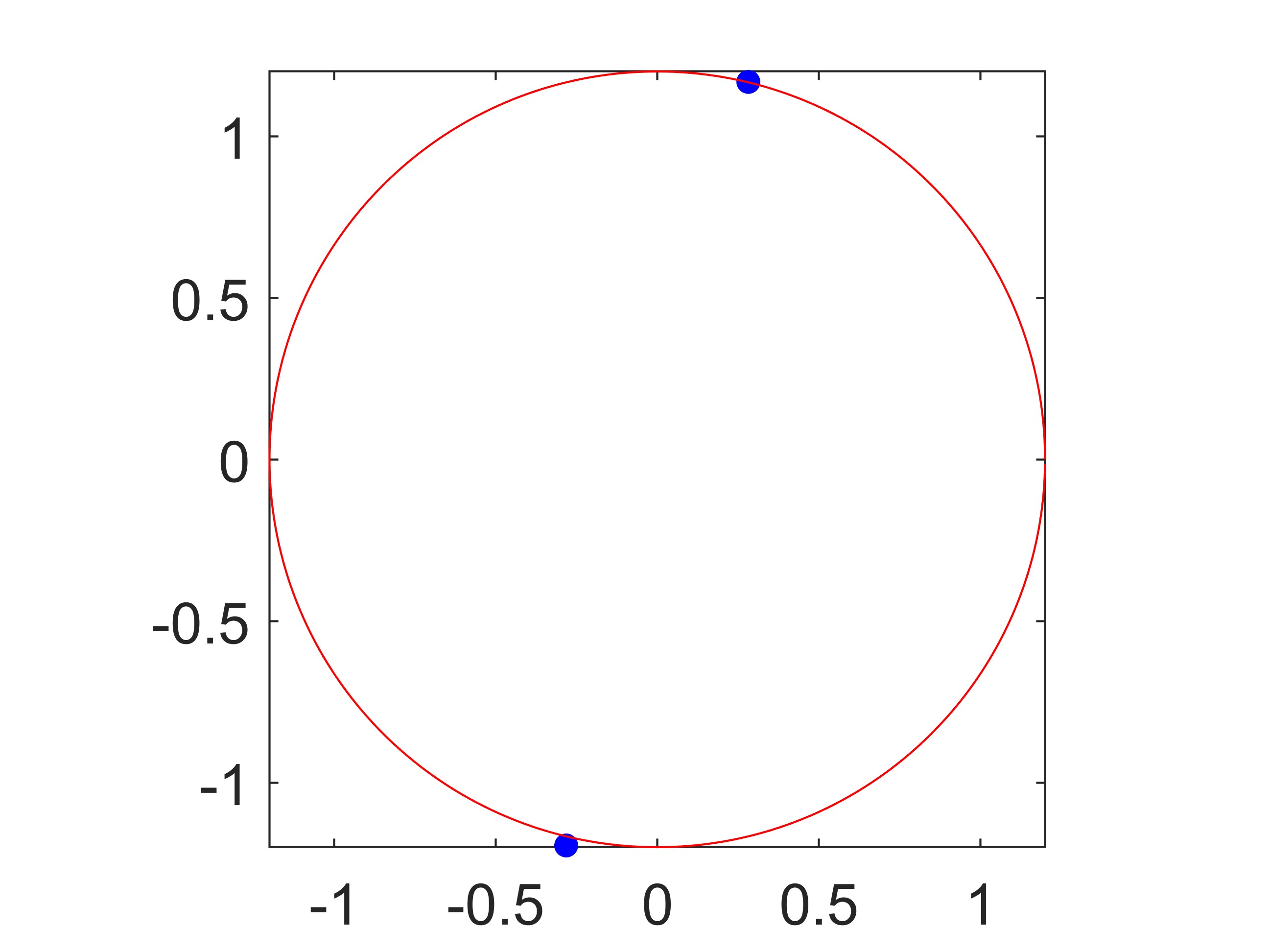}}\\
	\subfigure[]{\includegraphics[width=0.32\linewidth]{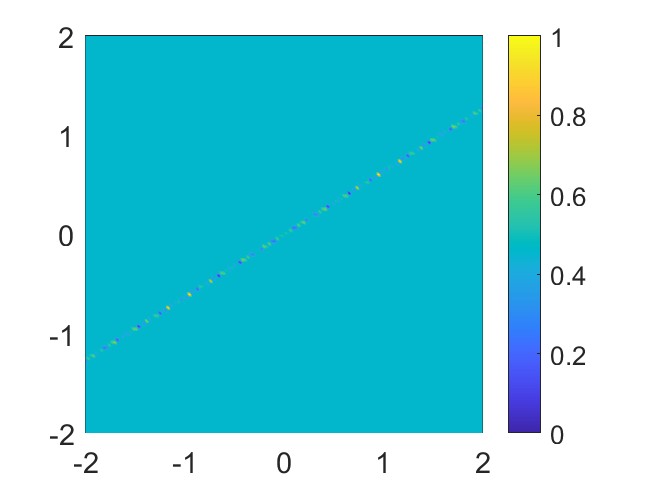}}
	\subfigure[]{\includegraphics[width=0.32\linewidth]{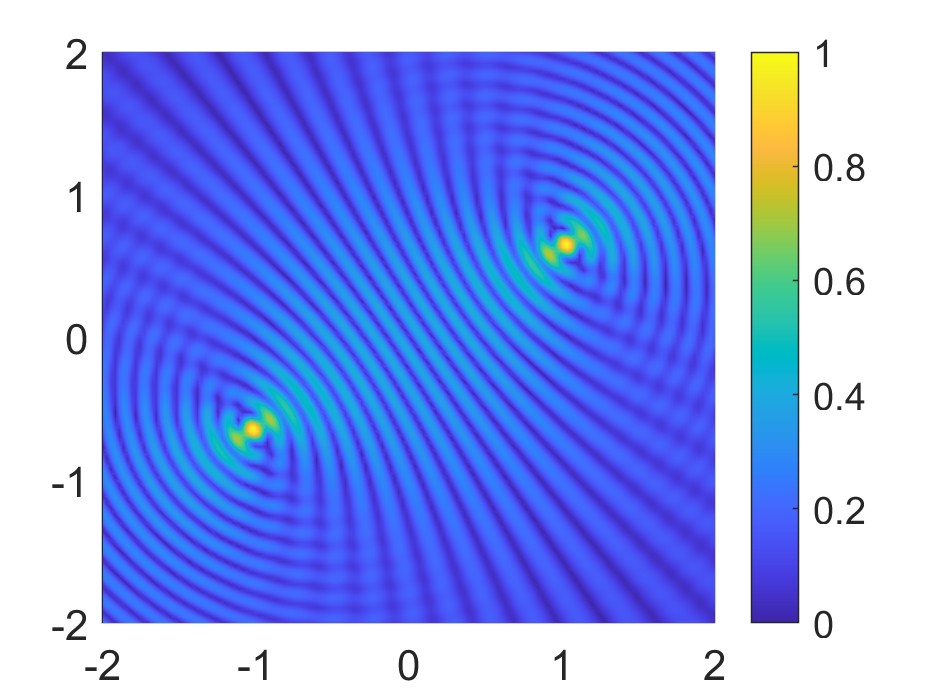}}
	\subfigure[]{\includegraphics[width=0.32\linewidth]{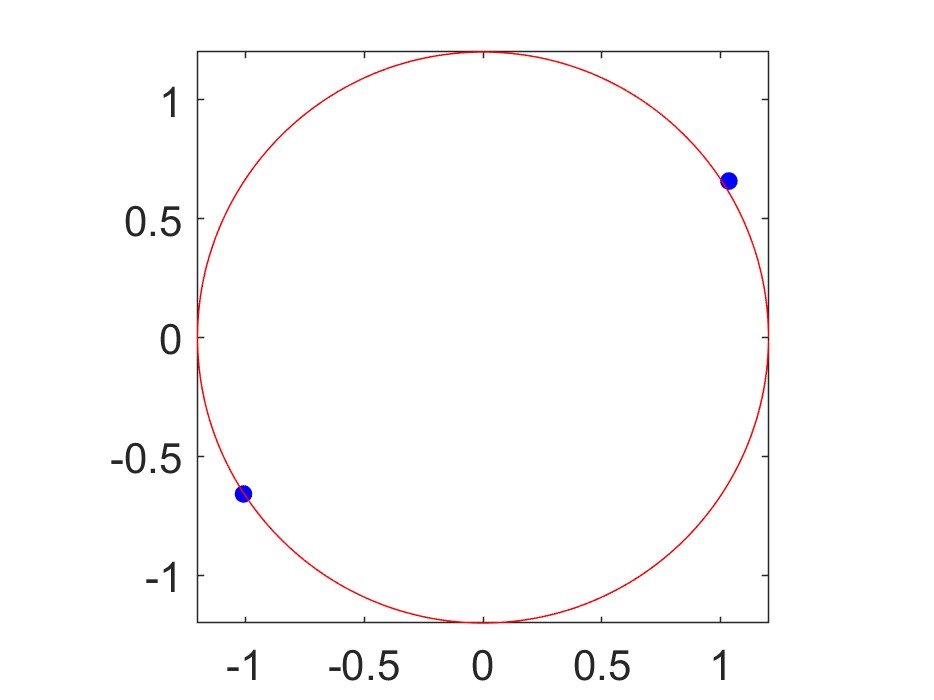}}
  \caption{The sampling process for a single incident direction. (The first row: $j=5$; the second row: $j=100$, the third row: $j=237,$ the last row: $j=430$).}\label{fig: circle_illustrate}
\end{figure}

\begin{example}
  In the second example, we test the performance of the proposed method by reconstructing the kite-shaped obstacle. The true boundary of the obstacle is given by the parameterized form
  $$
  x(t)=(\cos t+0.65\cos 2t-0.65, 1.5\sin t),\quad t\in[0,2\pi).
  $$
  
  Typically, the recovery of the kite-shaped obstacle is regarded as a benchmark for the quality of a reconstruction algorithm. Particularly, a major concern is whether the concave part of the obstacle can be successfully recovered. In this example, we take the parameters as $N_R=256, g=5\times10^{-4},$ and $M=2$.
  
  By respectively taking $N_d$ to be $128, 256, 512, 1024$ and 2048, we exhibit the reconstructions for $k=20$ in \Cref{fig: kite1}. From \Cref{fig: kite1}, we observe that the obstacle is well reconstructed, except for the short arcs close to the two wings. Further, the reconstruction can be improved using relatively more incident directions. A reason accounting for this is that the boundary is determined point by point, and each incident wave incurs the corresponding scattered field and thus determines a local point or several local points on the boundary. As a result, more incident waves could cover more points to approximate the boundary, which improves the final reconstruction.
  
  Next, we test the performance of the proposed method by recovering the kite-shaped obstacle from limited-aperture observation. Now we choose $N_d=256,N_R=512\left[\frac{2\pi}{\theta_{\max}}\right],$ with $\theta_{\max}\in(0,2\pi]$ being the observation aperture. We demonstrate the reconstructions with $k=20$ in \Cref{fig: kite2}. In \Cref{fig: kite2}, we utilize the black dashed line to represent the observed aperture, and the green dashed line to facilitate the visualization of the aperture angle.
  We can observe from \Cref{fig: kite2} that the full boundary of the obstacle can be well-reconstructed when the observation aperture is relatively large, such as $3\pi/2,$ or even $\pi.$ 
  Nevertheless, for a smaller observation aperture such as $\frac{\pi}{2},$ it can be found from \Cref{fig: kite2}(e)--\Cref{fig: kite2}(f) that only the illuminated part is well-reconstructed while the shadow domain may not be well identified due to the lack of information.
\end{example}

\begin{figure}
	\centering
	\subfigure[model setup]{\includegraphics[width=0.3\linewidth]{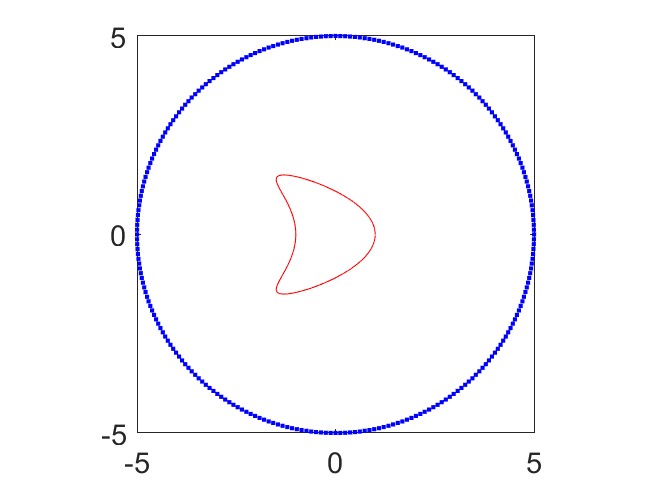}}
    \subfigure[$N_d=128$]{\includegraphics[width=0.3\linewidth]{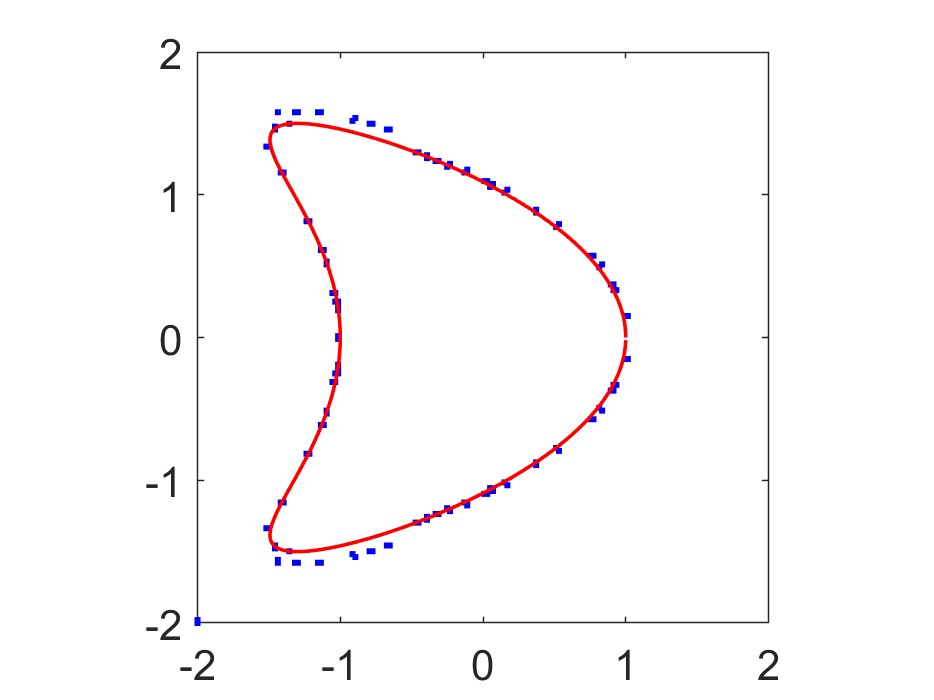}}
    \subfigure[$N_d= 256$]{\includegraphics[width=0.3\linewidth]{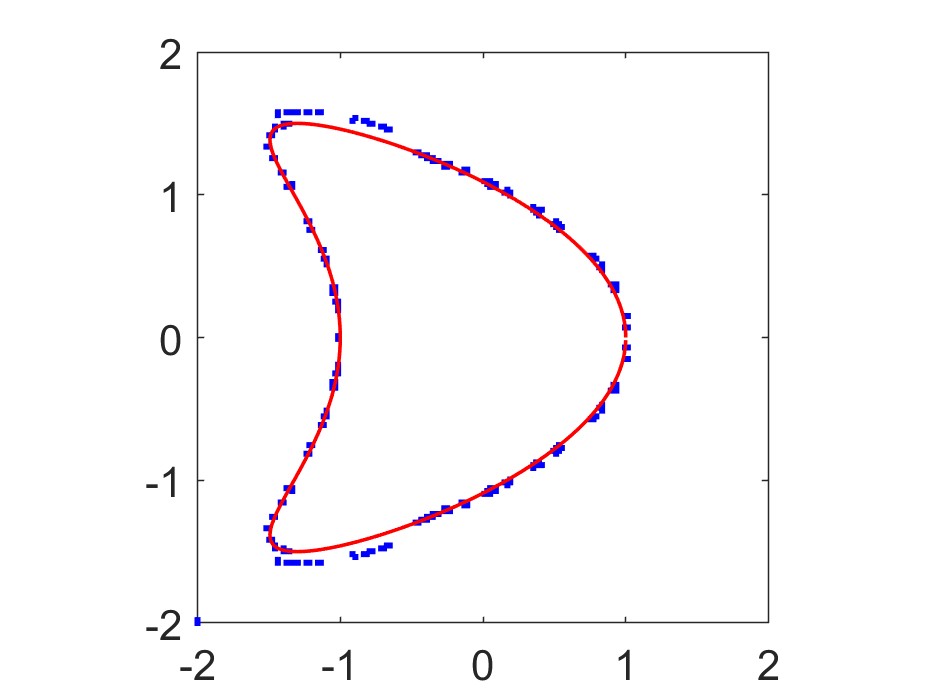}}\\
    \subfigure[$N_d= 512$]{\includegraphics[width=0.3\linewidth]{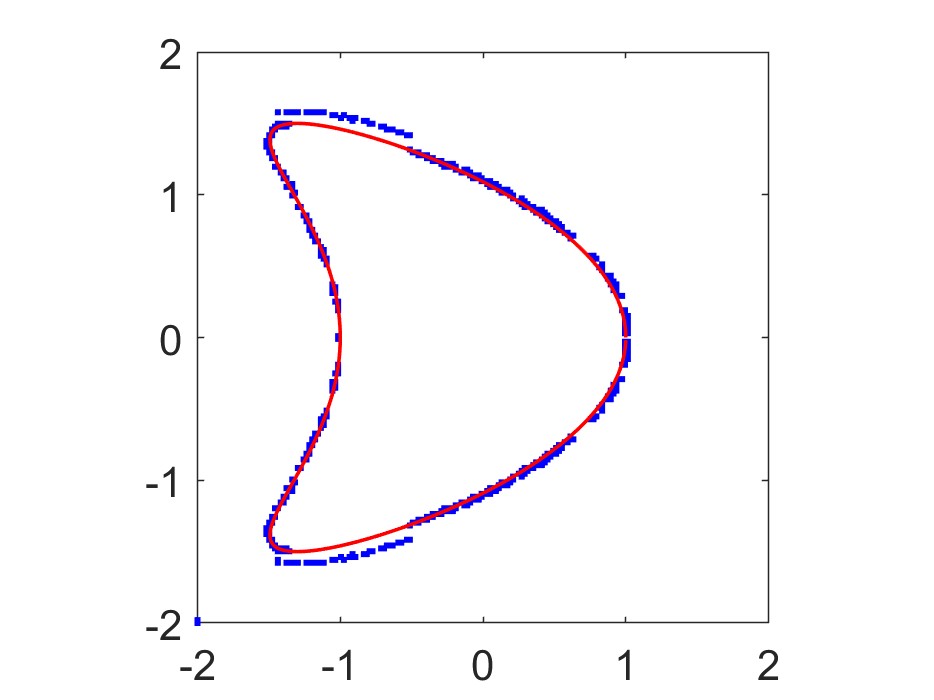}}
    \subfigure[$N_d= 1024$]{\includegraphics[width=0.3\linewidth]{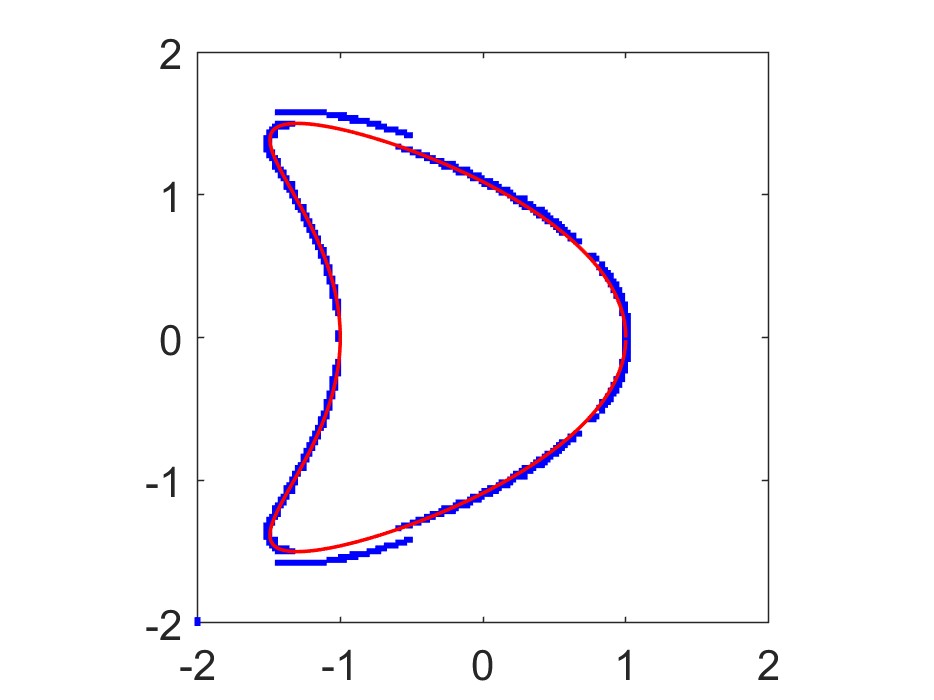}}
    \subfigure[$N_d= 2048$]{\includegraphics[width=0.3\linewidth]{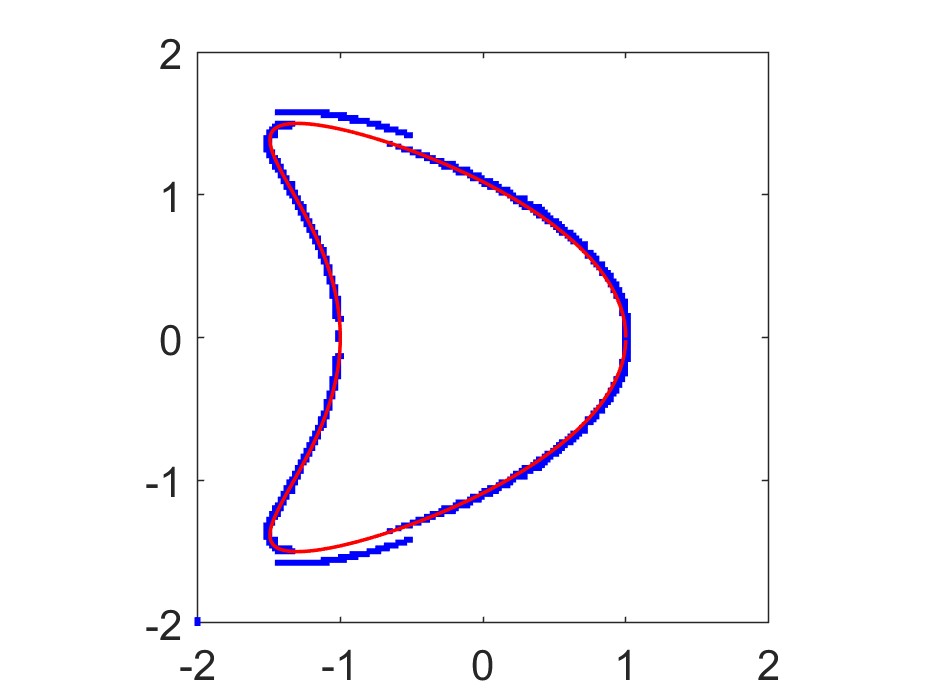}}
  \caption{Reconstruction of the kite-shaped obstacle for $k=20.$}\label{fig: kite1}
\end{figure}

\begin{figure}
    \centering
    \subfigure[]{\includegraphics[width=0.3\linewidth]{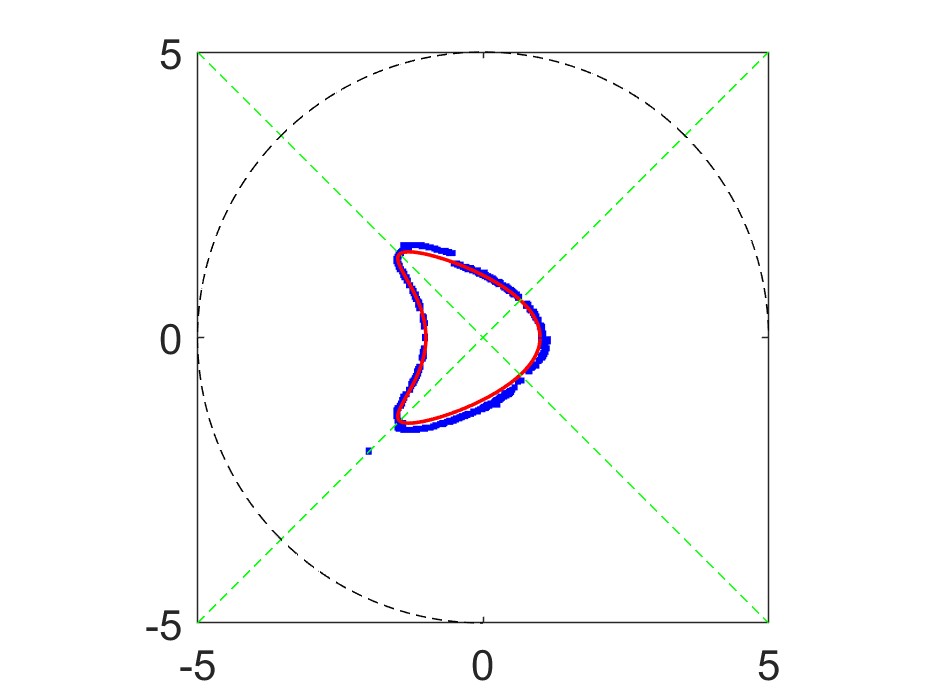}} 
    \subfigure[]{\includegraphics[width=0.3\linewidth]{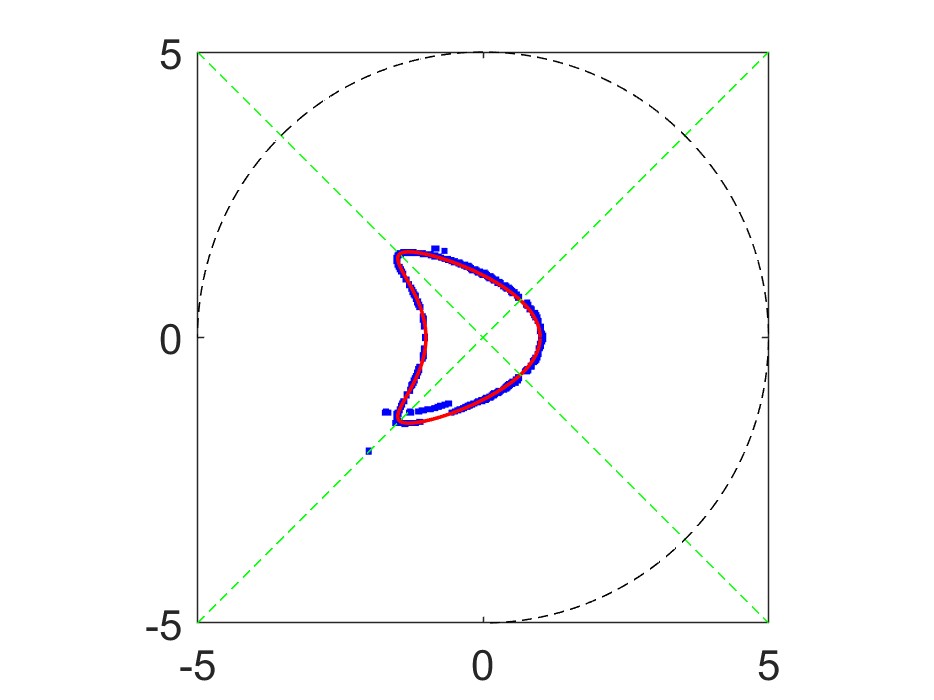}}
    \subfigure[]{\includegraphics[width=0.3\linewidth]{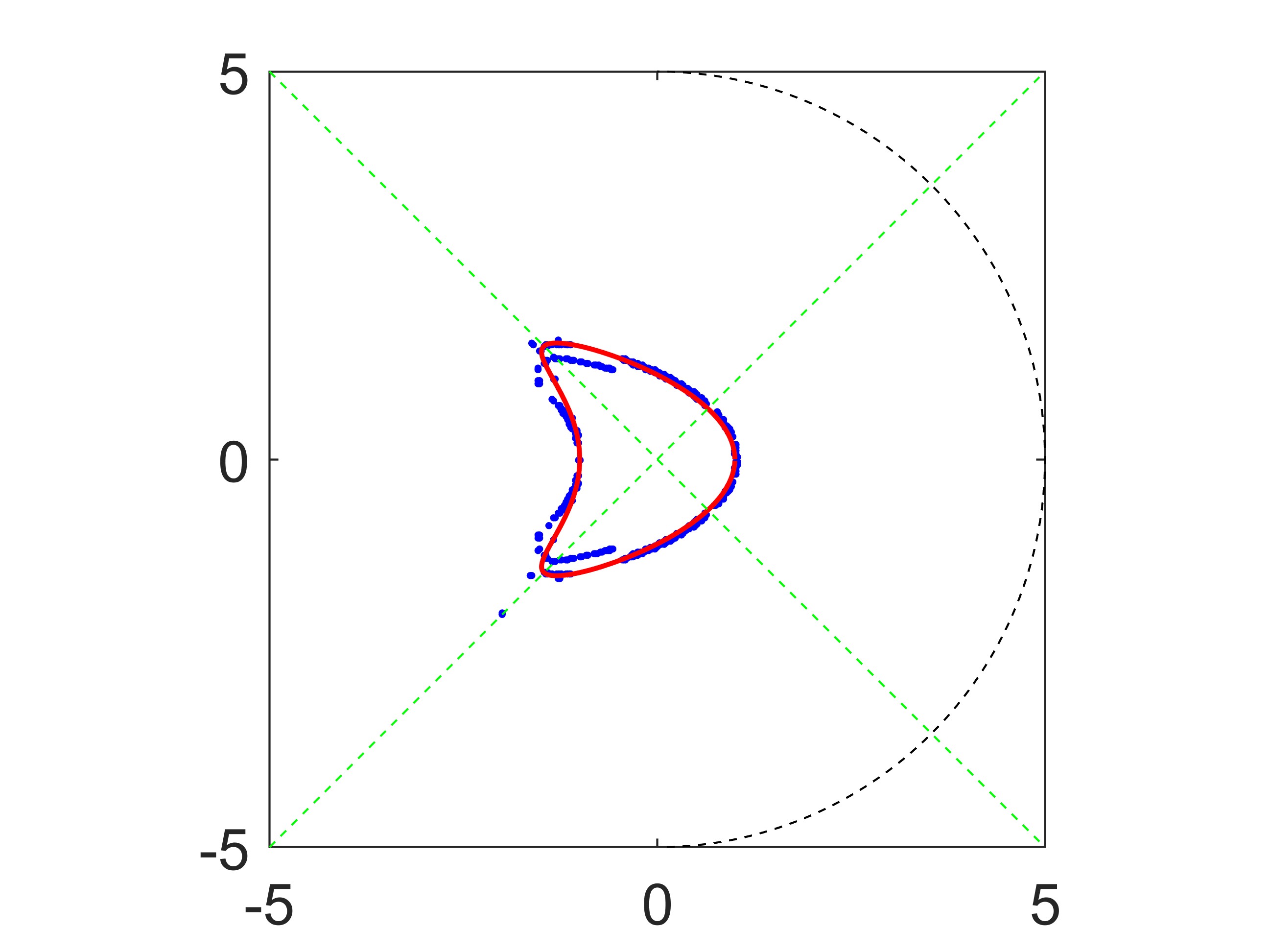}}\\
    \subfigure[]{\includegraphics[width=0.3\linewidth]{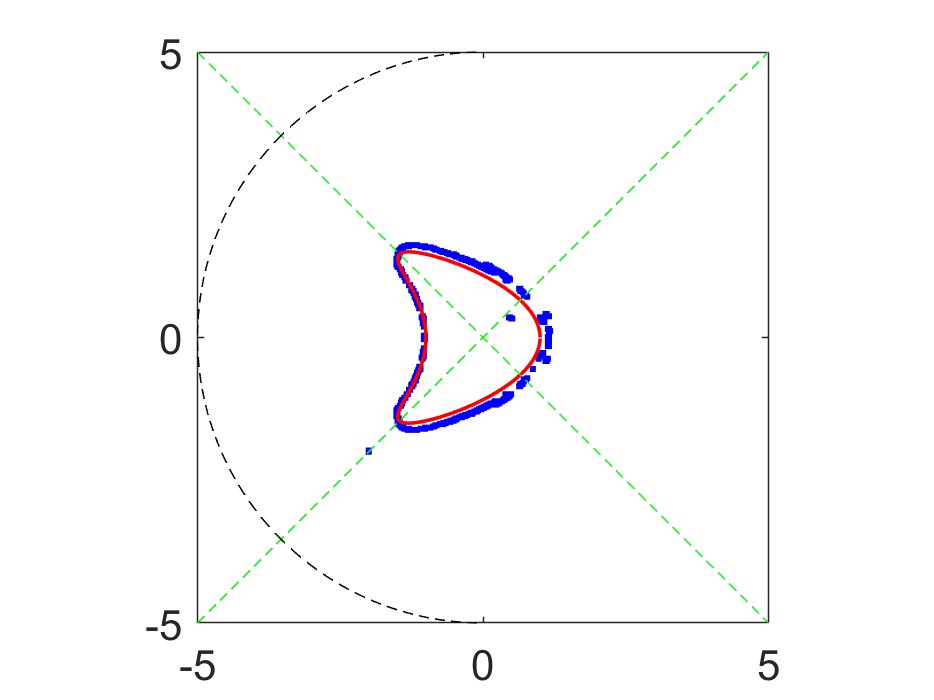}}
    \subfigure[]{\includegraphics[width=0.3\linewidth]{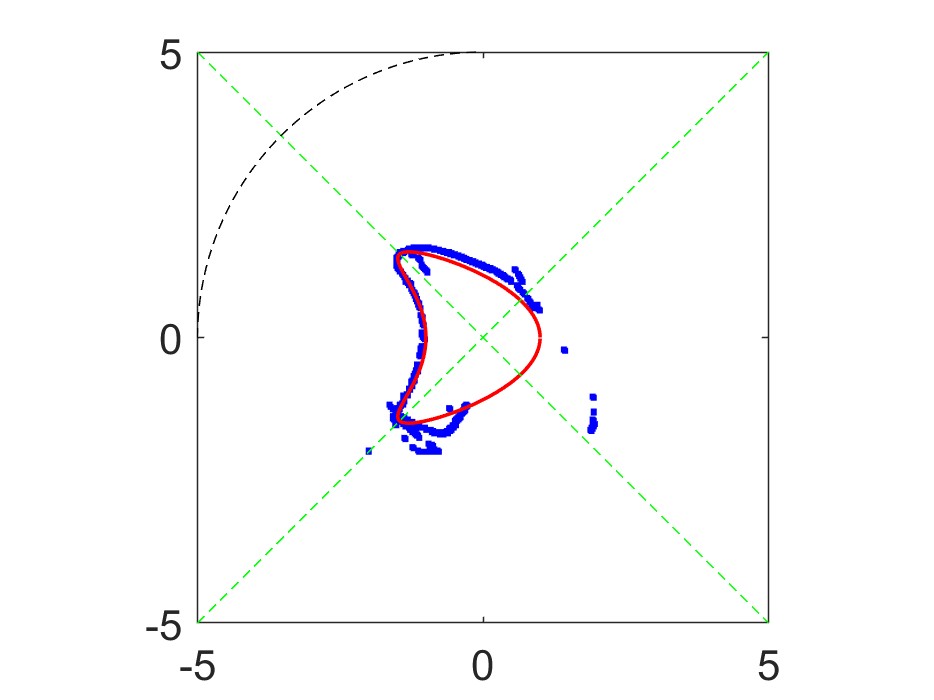}}
    \subfigure[]{\includegraphics[width=0.3\linewidth]{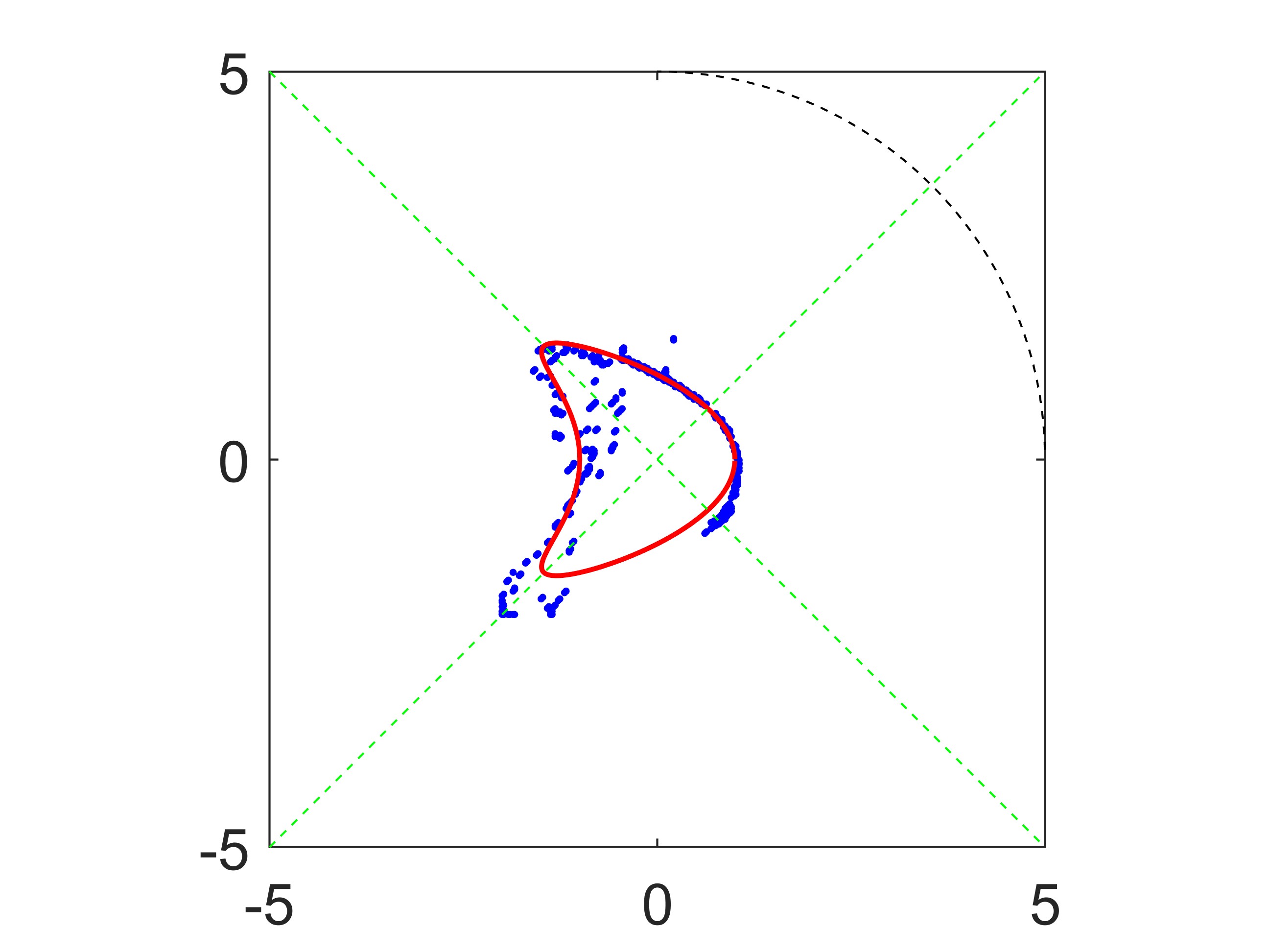}}
  \caption{Reconstruction of the kite-shaped obstacle for $k=20.$}\label{fig: kite2}
\end{figure}

\begin{example}
  In this example, we are devoted to reconstructing the obstacle with L-leaf components. The boundary is parameterized by 
  \[
  x(t)=(1+0.2\cos Lt)(\cos t,\sin t),\quad L=3,4,5.
  \]
  
  By taking $k=25, g=3\times10^{-3},$ we show the reconstructions for $N_d=1024,\,N_R=128$ in \Cref{fig: Lleaf}.
  These results imply that the method performs well for imaging both the convex and concave parts of the obstacle. In addition, our imaging scheme is insensitive to the noise level. No matter the noise level $\delta$ is $5\%, 10\%$, or even $20\%,$ our method can always obtain an accurate reconstruction. These results illustrate that by choosing the incident field to be the tapered wave, the sampling-based method is capable of outputting the quantitative results. 
  
However, we would like to point out that the value of $g$ plays an important role in the reconstruction. By varying the values of $g$, we display the reconstructions for $k=25,$ $L=3$ and $\delta=5\%$ in \Cref{fig: pear}. We readily observe that the reconstructions are quite good when $g$ is relatively small. However, as the value of $g$ increases, the reconstruction may be detored and there may be a few errorous boundary-point recoveries. Nevertheless, the overall boundary can be well-matched and the results are acceptable.

Further, we point out that, to better reconstruct the obstacle, we can adopt a separated-domain sampling strategy as follows: 
\begin{itemize}
  \item In the first step, we divide the sampling domain into two subdomains, such that  
  $$
  \Omega=\Omega_1\cup \Omega_2,
  $$
  with $\Omega_1=[-2,2]\times[-2,0],\,\Omega_2=[-2,2]\times[0,2].$
  \item For $\ell=1,2,$ we compute the indicator function \eqref{Indicator} in these two domains, respectively. By collecting the first 512 maximum points of the indicator function in $\Omega_1$ and $\Omega_2$, we regard these two group maximum points as the approximation to the obstacle above and below the $x$-axis respectively. 
\end{itemize}

Aided by this decomposition approach, we further display the reconstruction in \Cref{fig: pear1}, where the corresponding parameters are the same as those in \Cref{fig: pear}.
Compared with \Cref{fig: pear}, we find that the discontinuous reconstruction on the boundary may be connected. We also note that this method works with the tradeoff of producing an additional line artifact on the interface of the sampling areas $\Omega_1$ and $\Omega_2.$
 
\end{example}

\begin{figure}
	\centering
	\subfigure[]{\includegraphics[width=0.32\linewidth]{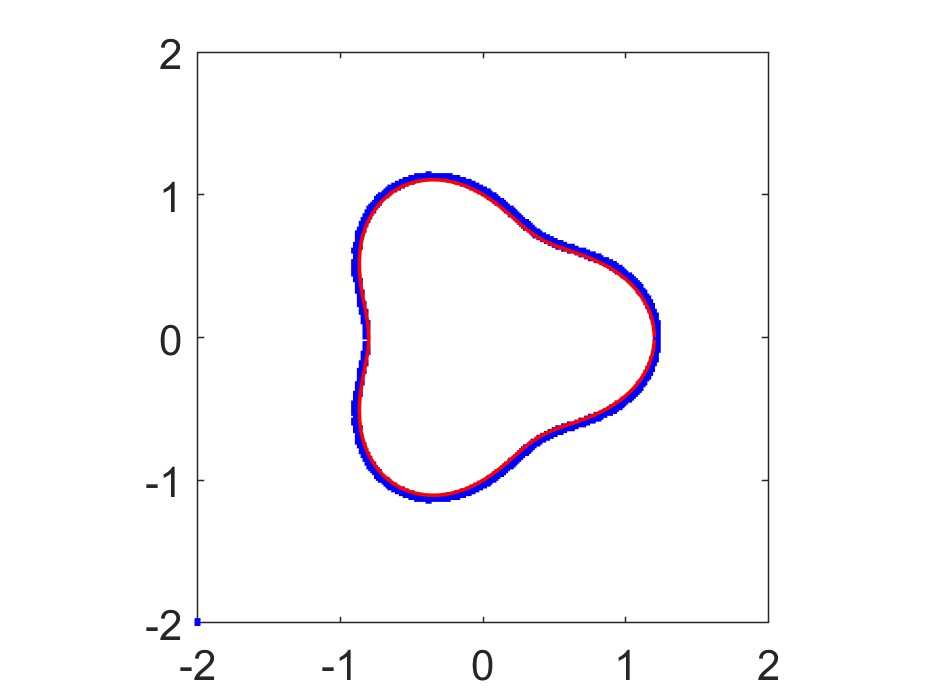}} 
	\subfigure[]{\includegraphics[width=0.32\linewidth]{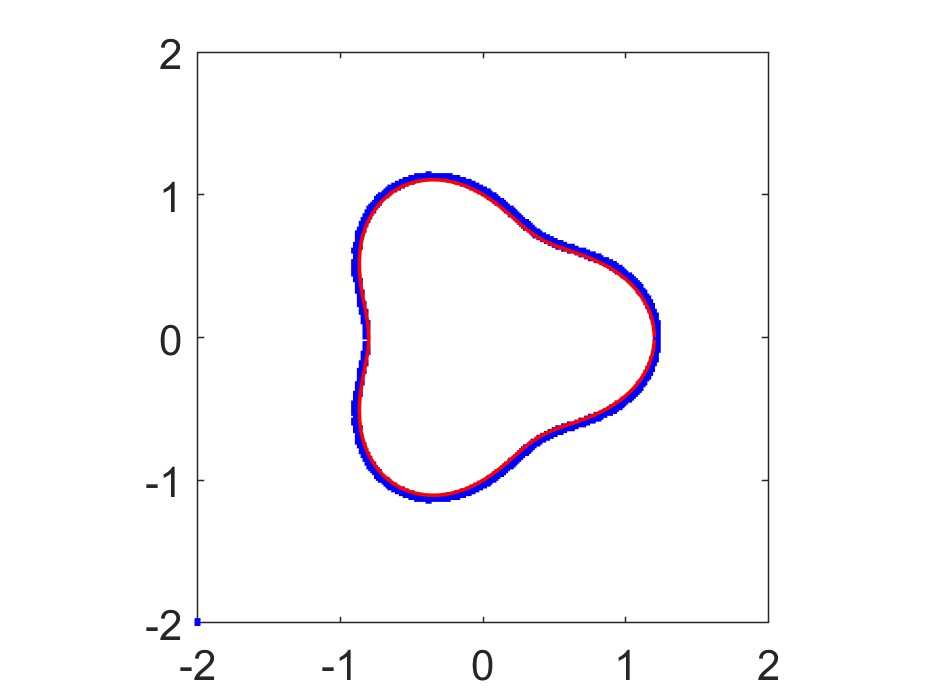}}
       \subfigure[]{\includegraphics[width=0.32\linewidth]{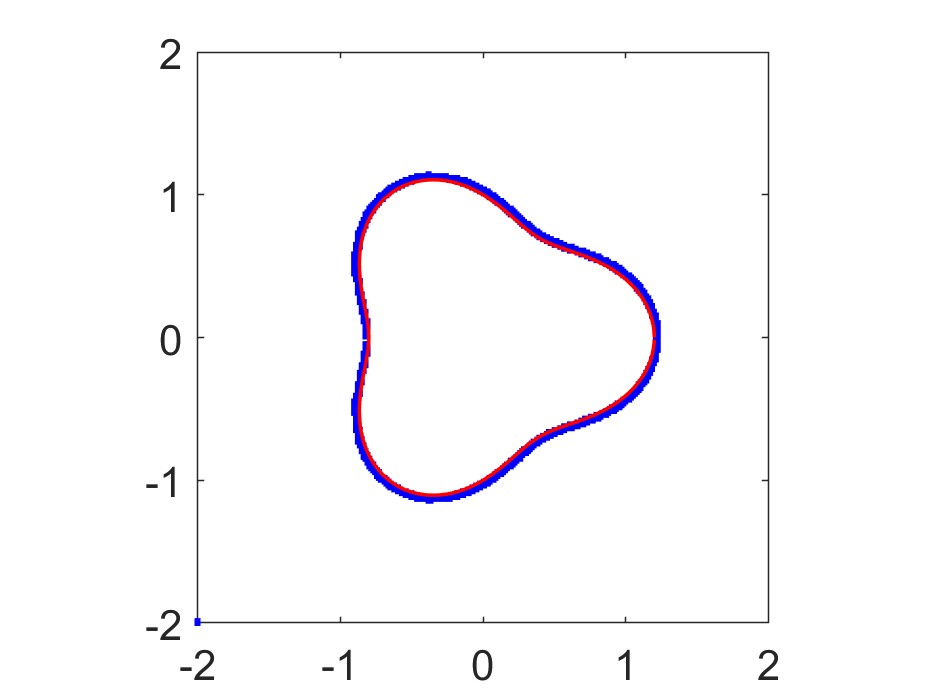}}\\
	\subfigure[]{\includegraphics[width=0.32\linewidth]{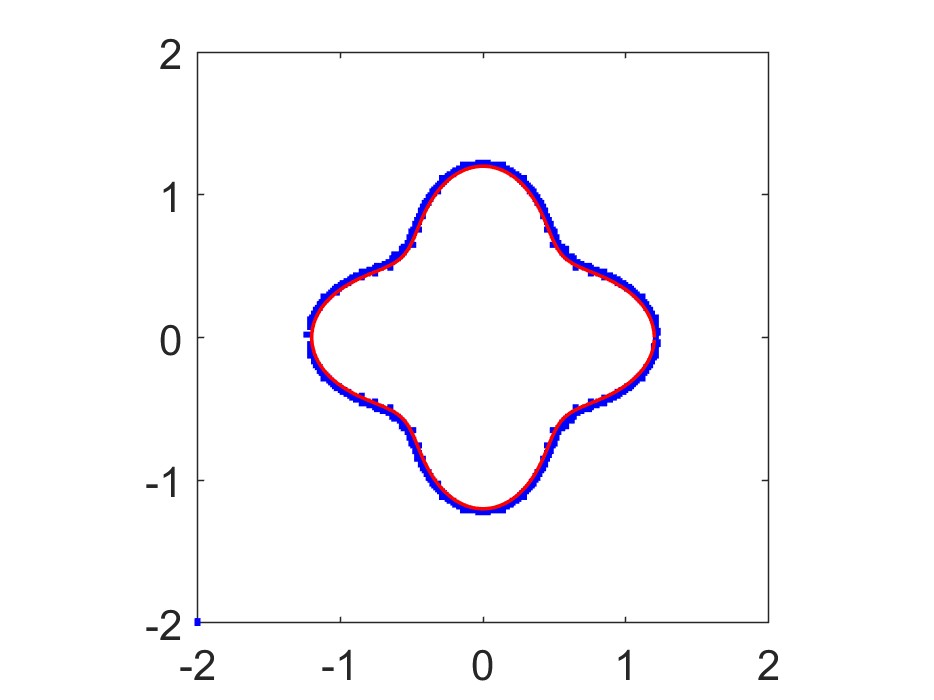}} 
	\subfigure[]{\includegraphics[width=0.32\linewidth]{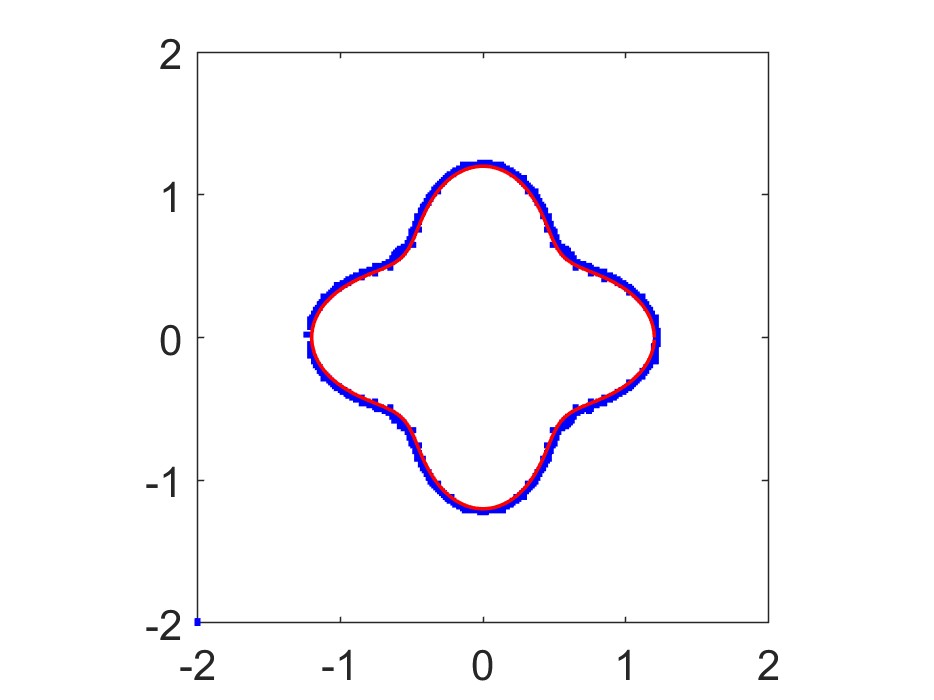}}
       \subfigure[]{\includegraphics[width=0.32\linewidth]{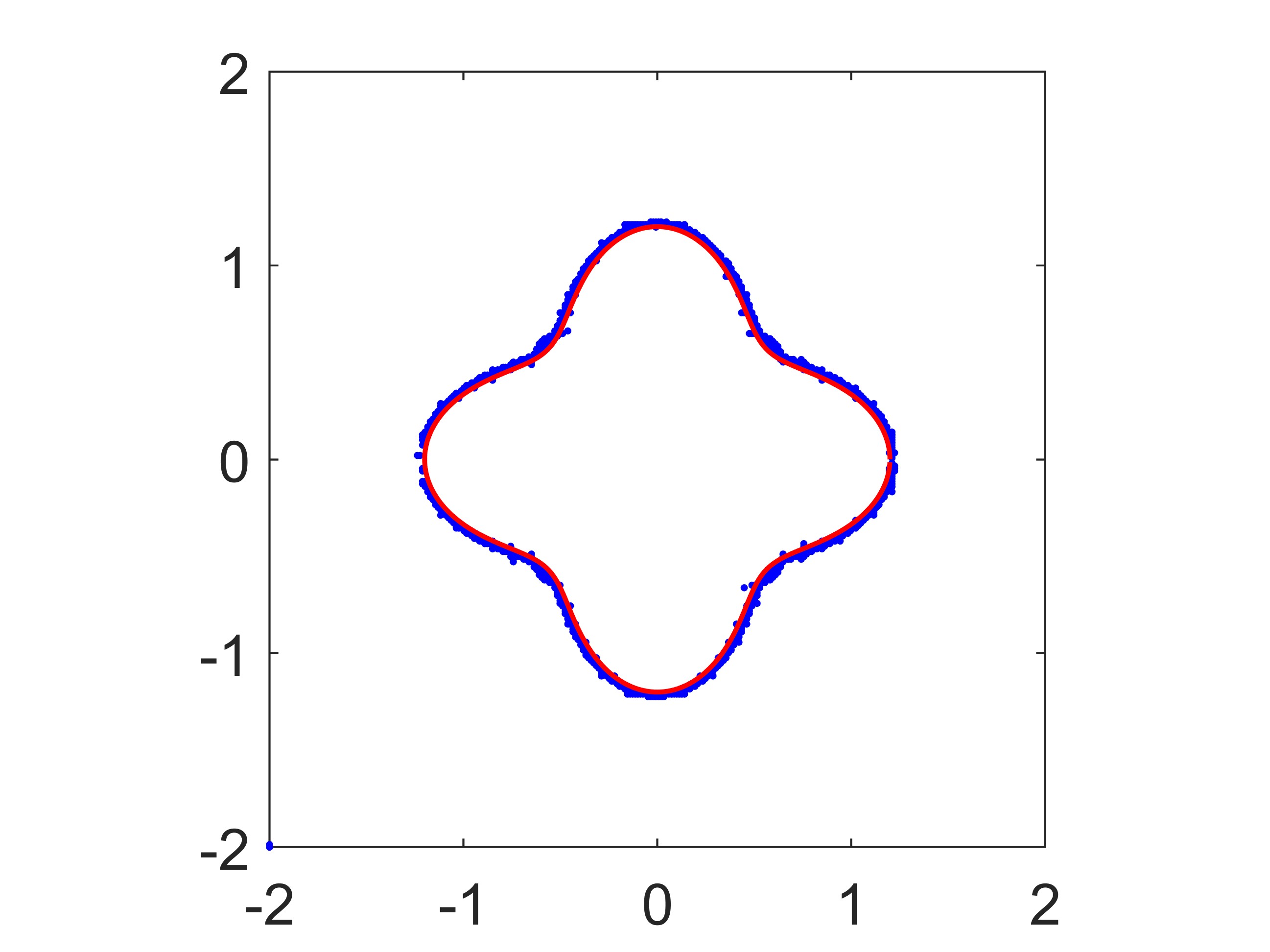}}\\
	\subfigure[]{\includegraphics[width=0.32\linewidth]{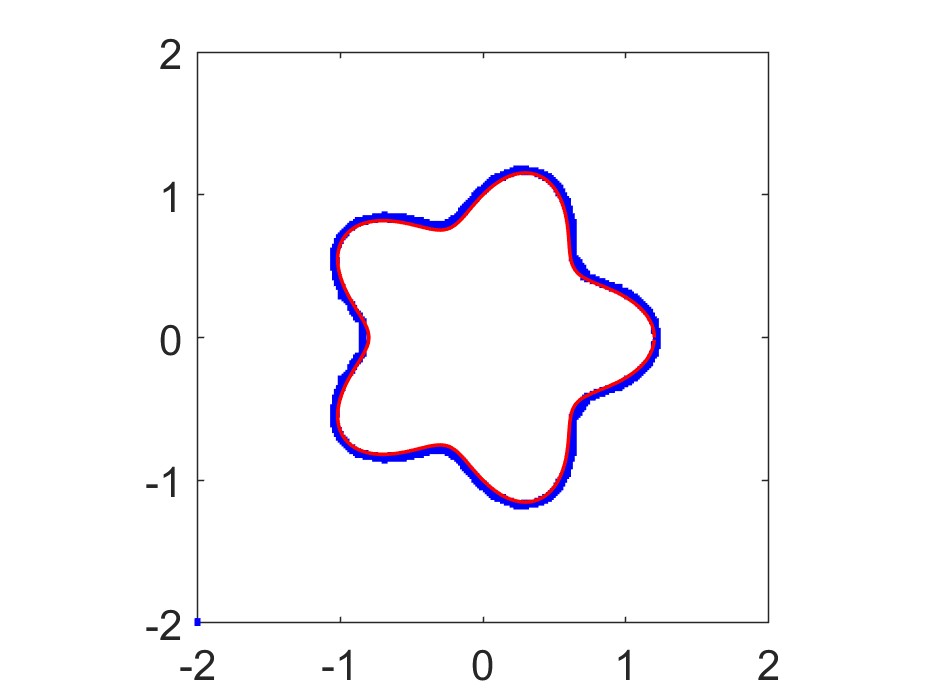}} 
	\subfigure[]{\includegraphics[width=0.32\linewidth]{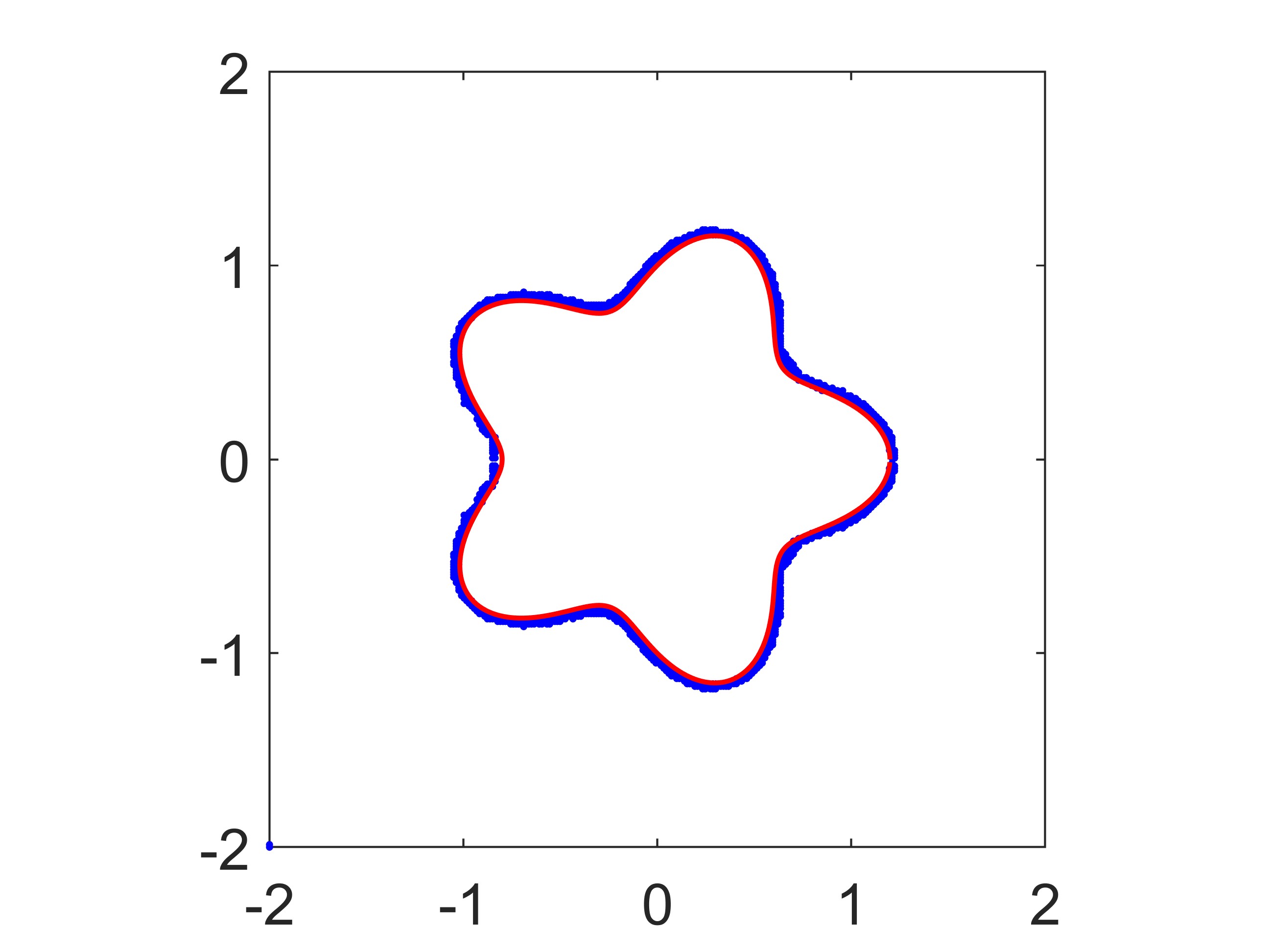}}
    \subfigure[]{\includegraphics[width=0.32\linewidth]{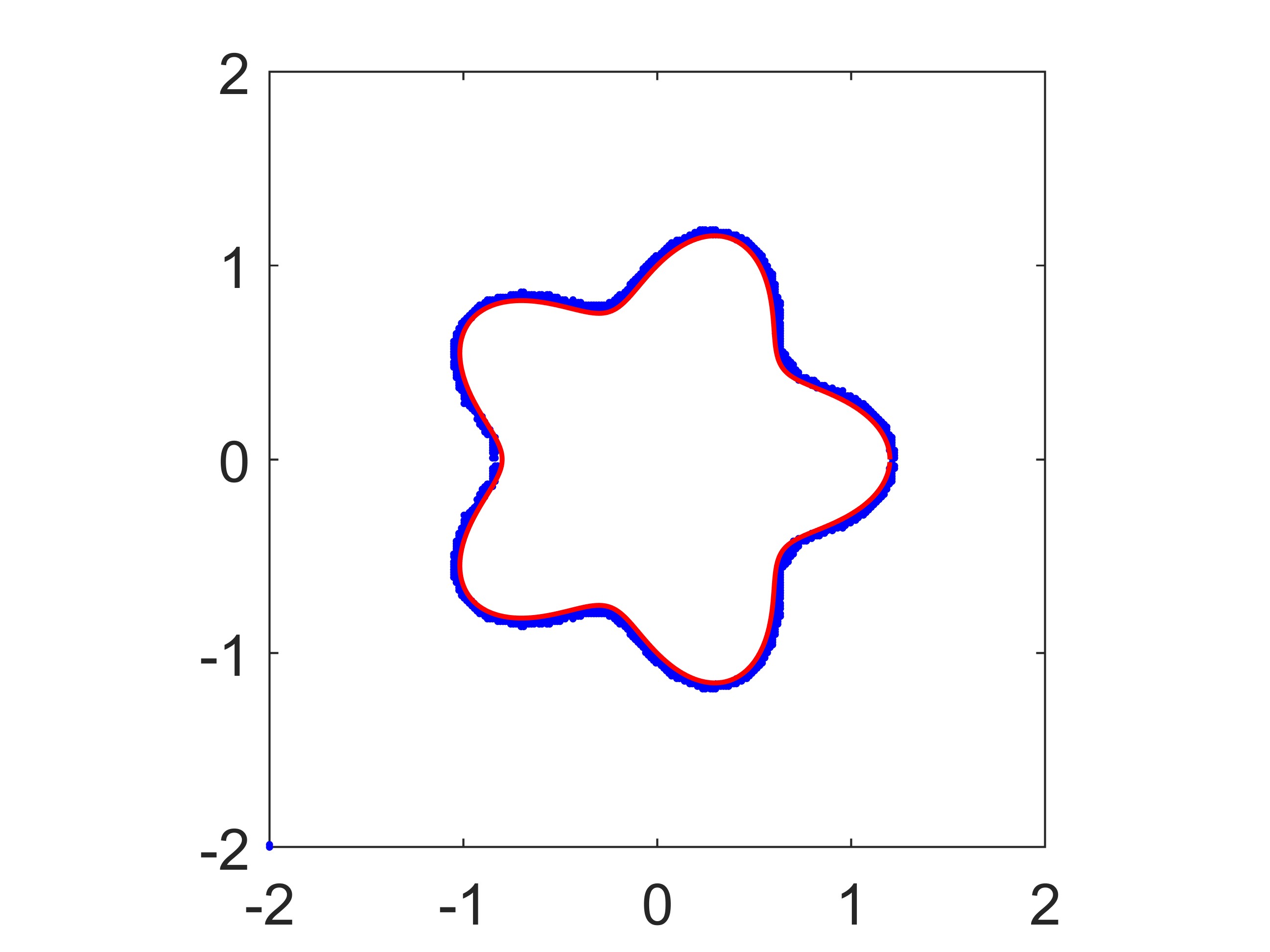}}\\
  \caption{Reconstruction of the L-leaf for $k=25$ (Left column: $\delta = 5\%;$ Middle column: $\delta = 10\%;$ Right column: $\delta = 20\%$).}\label{fig: Lleaf}
\end{figure}

\begin{figure}
    \centering
    \subfigure[$g=0.1$]{\includegraphics[width=0.32\linewidth]{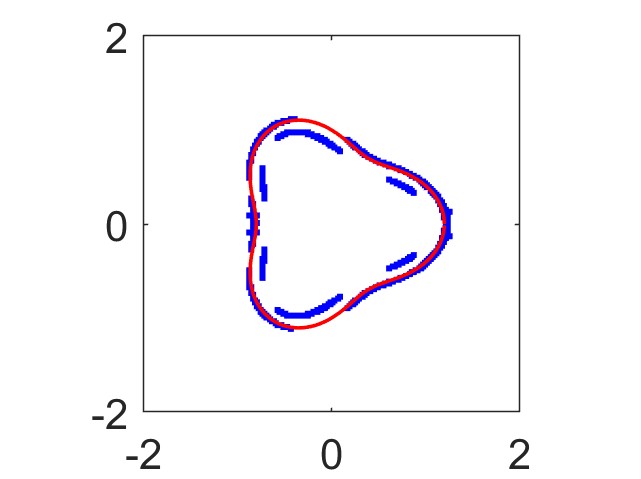}} 
    \subfigure[$g=0.05$]{\includegraphics[width=0.32\linewidth]{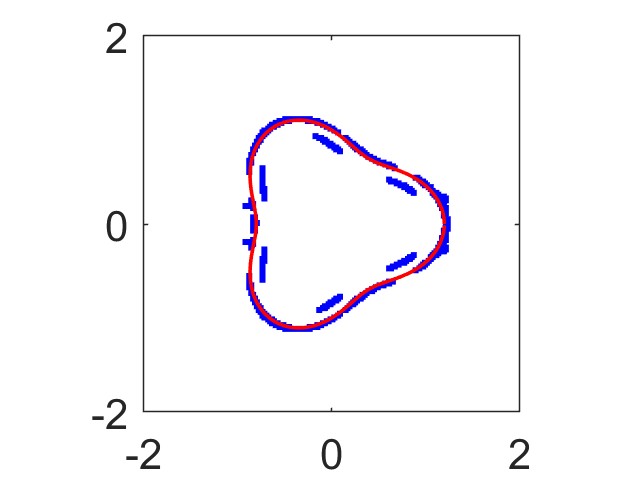}}
    \subfigure[$g=0.01$]{\includegraphics[width=0.32\linewidth]{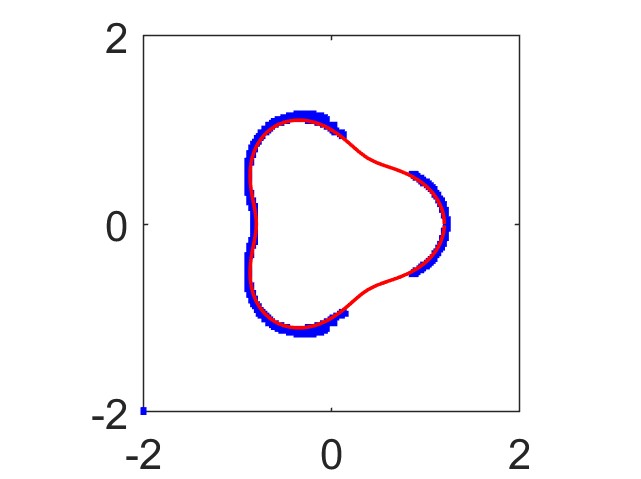}}
    \subfigure[$g=0.005$]{\includegraphics[width=0.32\linewidth]{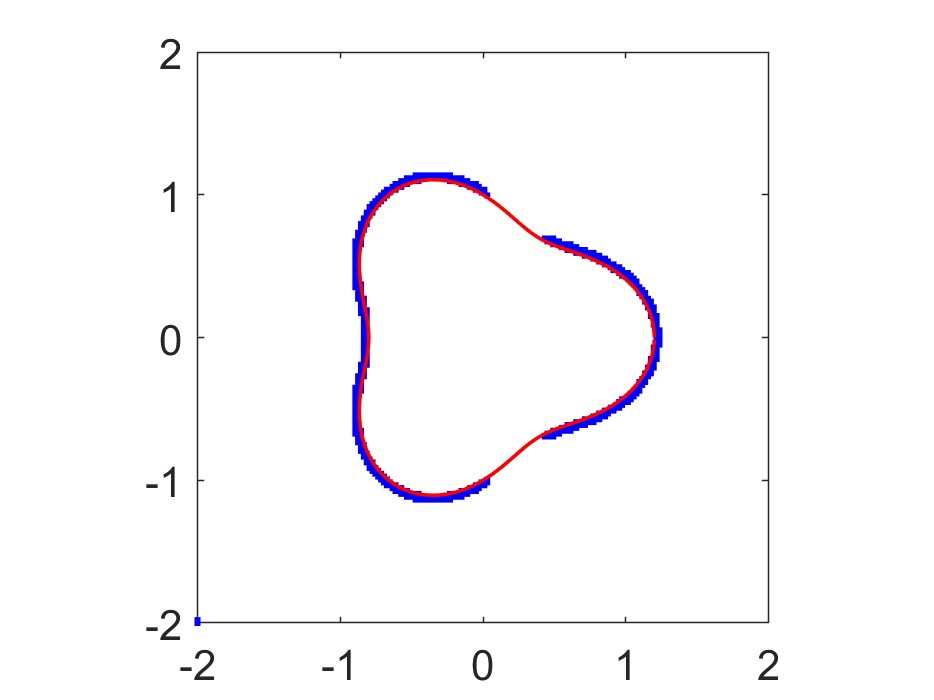}}
    \subfigure[$g=0.003$]{\includegraphics[width=0.32\linewidth]{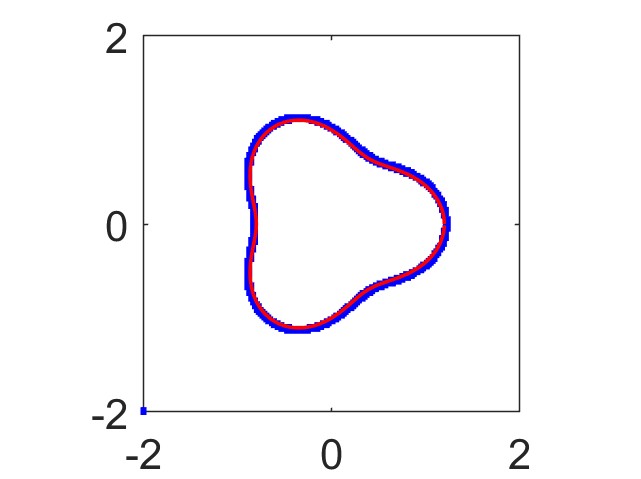}}
    \subfigure[$g=0.001$]{\includegraphics[width=0.32\linewidth]{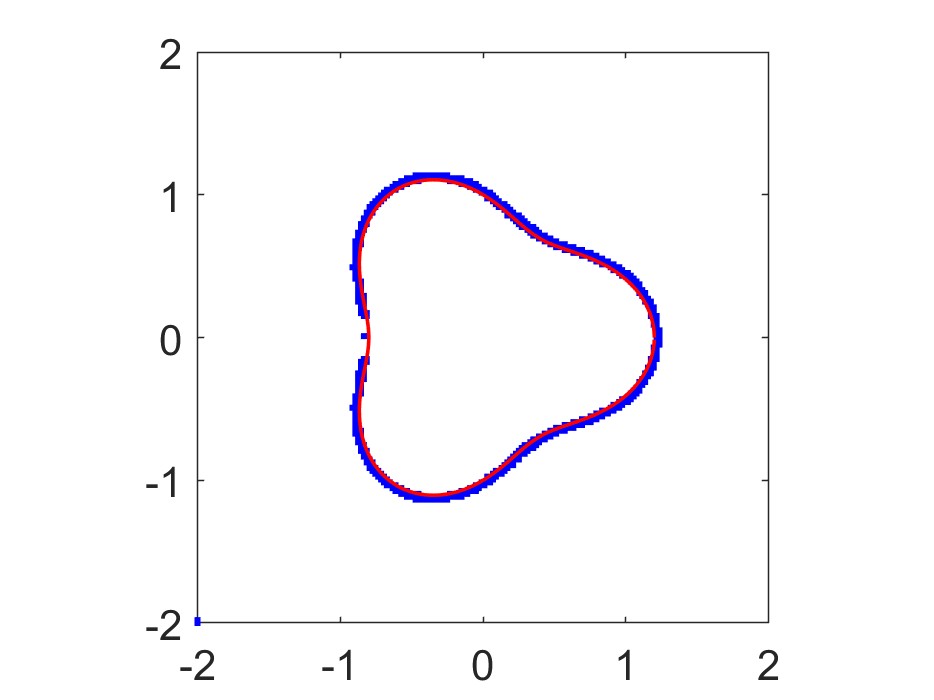}}
  \caption{Reconstruction of the pear for $k=25$ under different $g.$ }\label{fig: pear}
\end{figure}

\begin{figure}
    \centering
    \subfigure[$g=0.1$]{\includegraphics[width=0.32\linewidth]{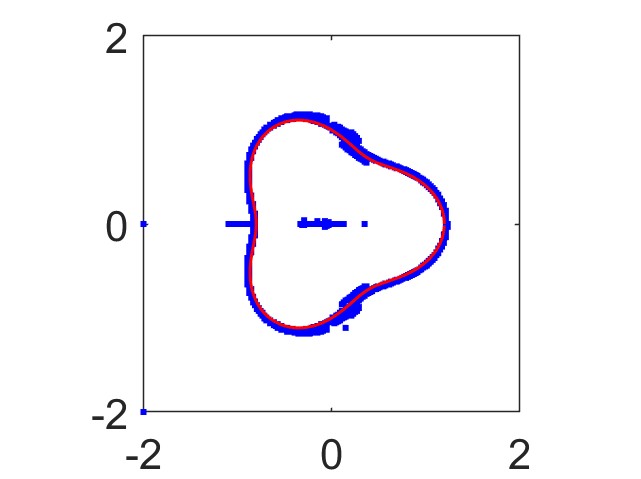}} 
    \subfigure[$g=0.05$]{\includegraphics[width=0.32\linewidth]{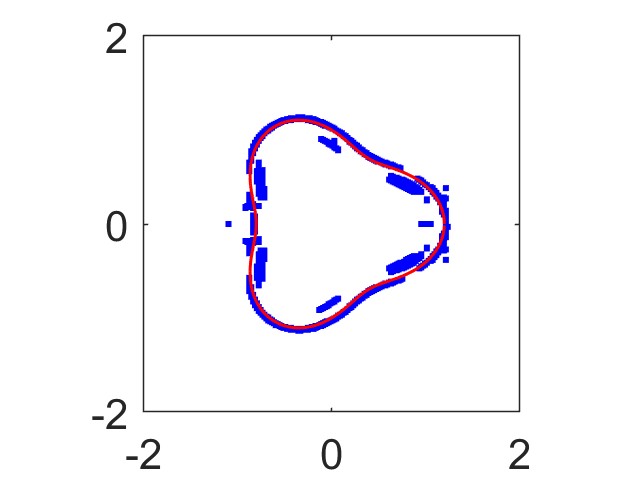}}
    \subfigure[$g=0.01$]{\includegraphics[width=0.32\linewidth]{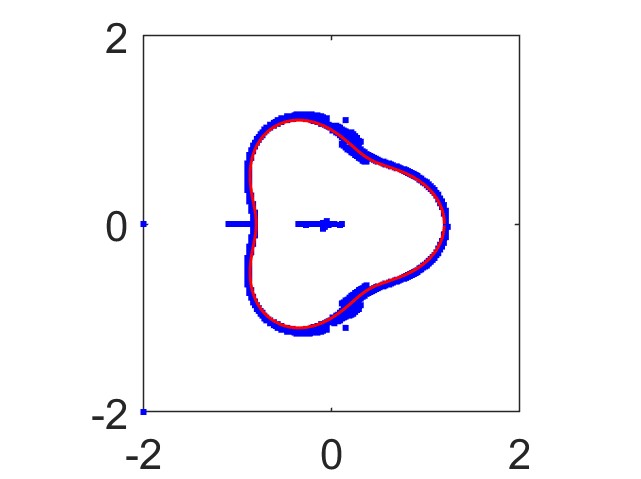}}
    \subfigure[$g=0.005$]{\includegraphics[width=0.32\linewidth]{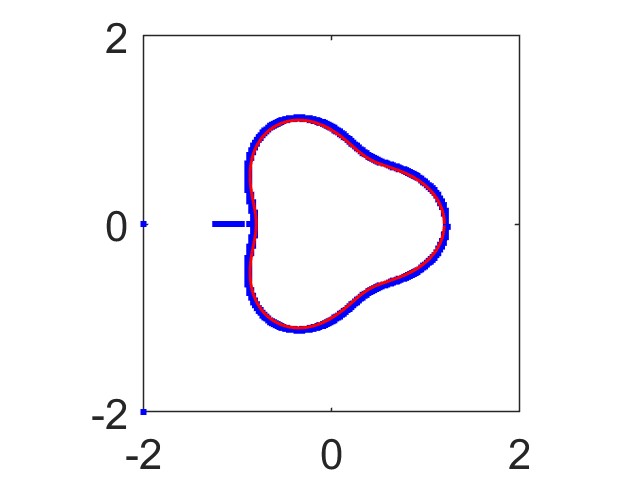}}
    \subfigure[$g=0.003$]{\includegraphics[width=0.32\linewidth]{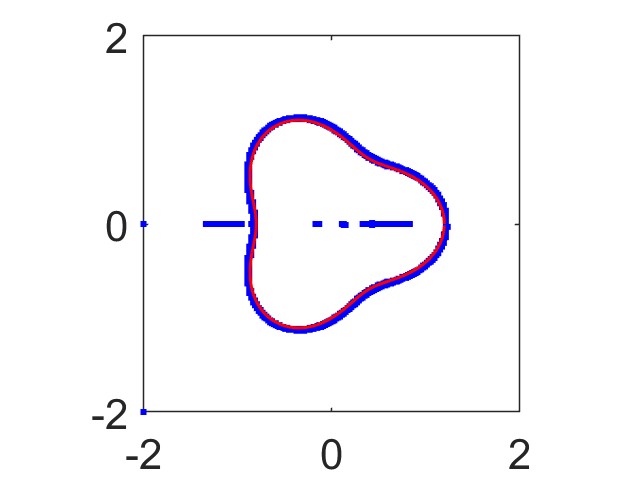}}
    \subfigure[$g=0.001$]{\includegraphics[width=0.32\linewidth]{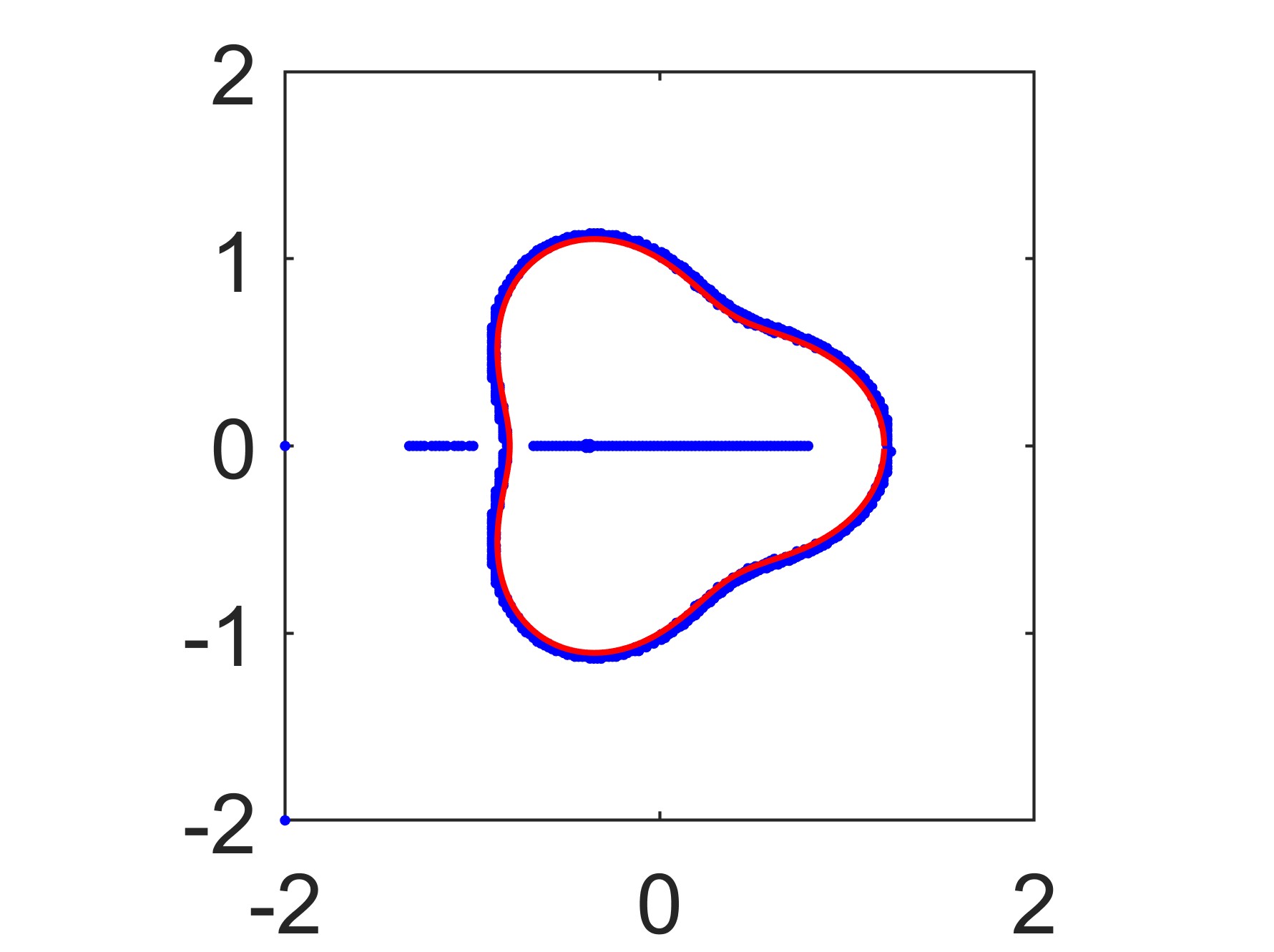}}
  \caption{Reconstruction of the pear for $k=25$ under separated domains. }\label{fig: pear1}
\end{figure}

\begin{example}
  As the last example, we aim to reconstruct the multiple obstacles from the scattered field. The three disjoint obstacles are respectively parameterized by 
  \begin{align*}
    \text{peanut:}&
    \quad x(t)=(-2,2)+\sqrt{3\cos^2t+1}\left(\cos \left(t+\frac{\pi}{4}\right),\sin \left(t+\frac{\pi}{4}\right)\right),\quad t\in[0,2\pi).\\
    \text{kite:}&\quad x(t)=(\cos t+0.65\cos2t-0.65,1.5\sin t-3),\quad t\in[0,2\pi).\\
    \text{pear:}&\quad x(t)=(3,2)+(1+0.15\cos3t)(\cos t,\sin t),\quad t\in[0,2\pi).
  \end{align*}
\end{example}

We plot the exact scatterers in \Cref{fig: model}. The parameters are taken to be $N_d=1024, N_R=512.$ The radius of the observation circle is chosen to be $10.$ By taking different $g,$ we display the reconstructions in \Cref{fig: multi}. These results show that all the disconnected components can be recovered. For the cases $g=0.05$ and $g=0.1$, the part located in the illuminated region can be better recovered while the ability to reconstruct the boundaries that are inadequately illuminated is limited.

\begin{figure}
  \centering
  \includegraphics[width=0.5\linewidth]{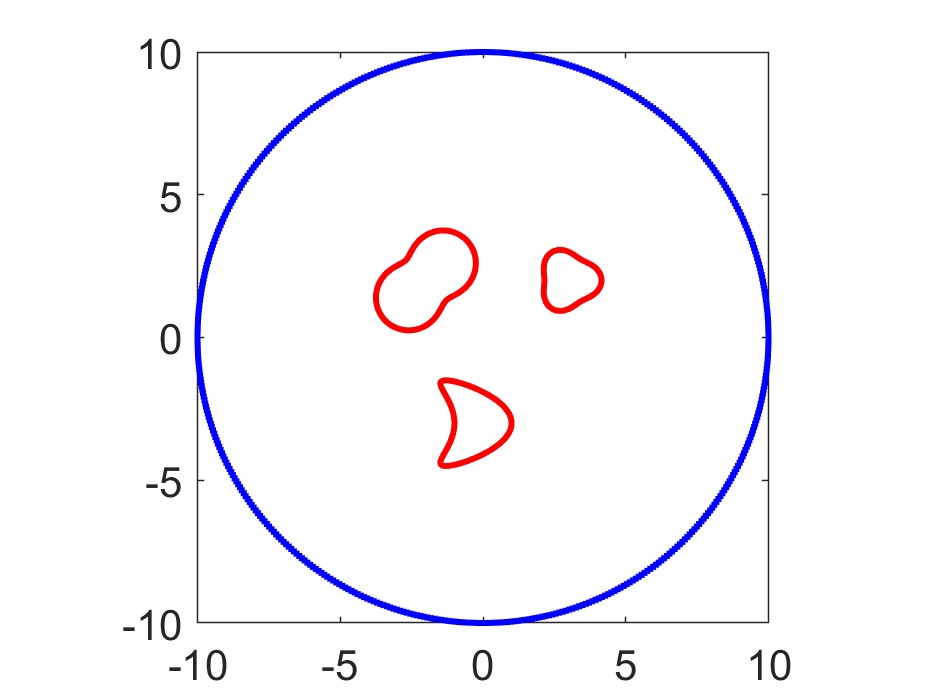}
  \caption{Exact scatterer that has three disjoint components.}\label{fig: model}
\end{figure}

\begin{figure}
  \centering
  \subfigure[$g=5\times10^{-4}$]{\includegraphics[width=0.3\linewidth]{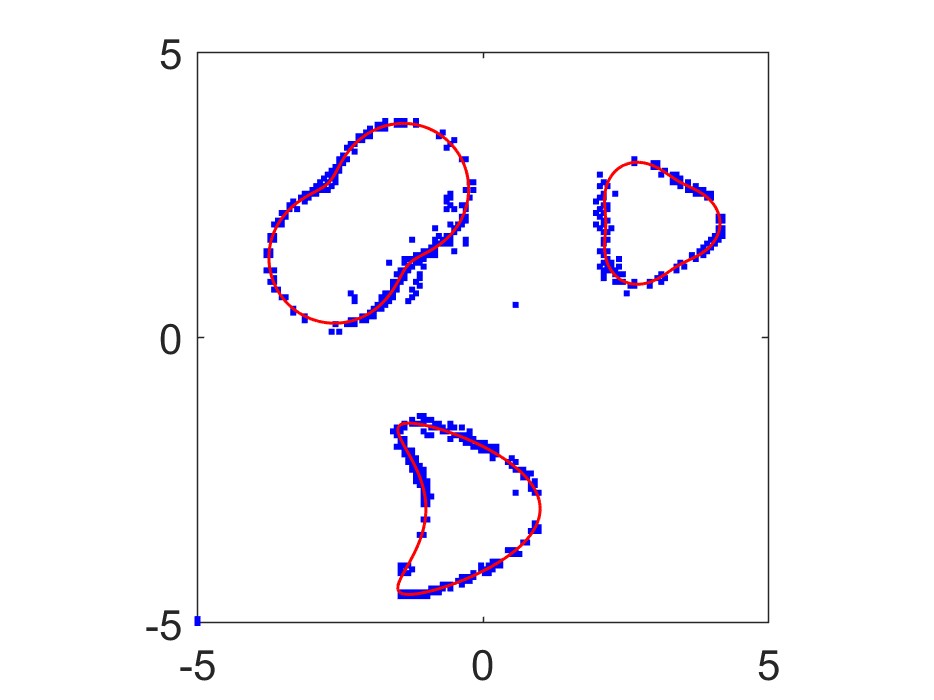}} 
  \subfigure[$g=0.005$]{\includegraphics[width=0.3\linewidth]{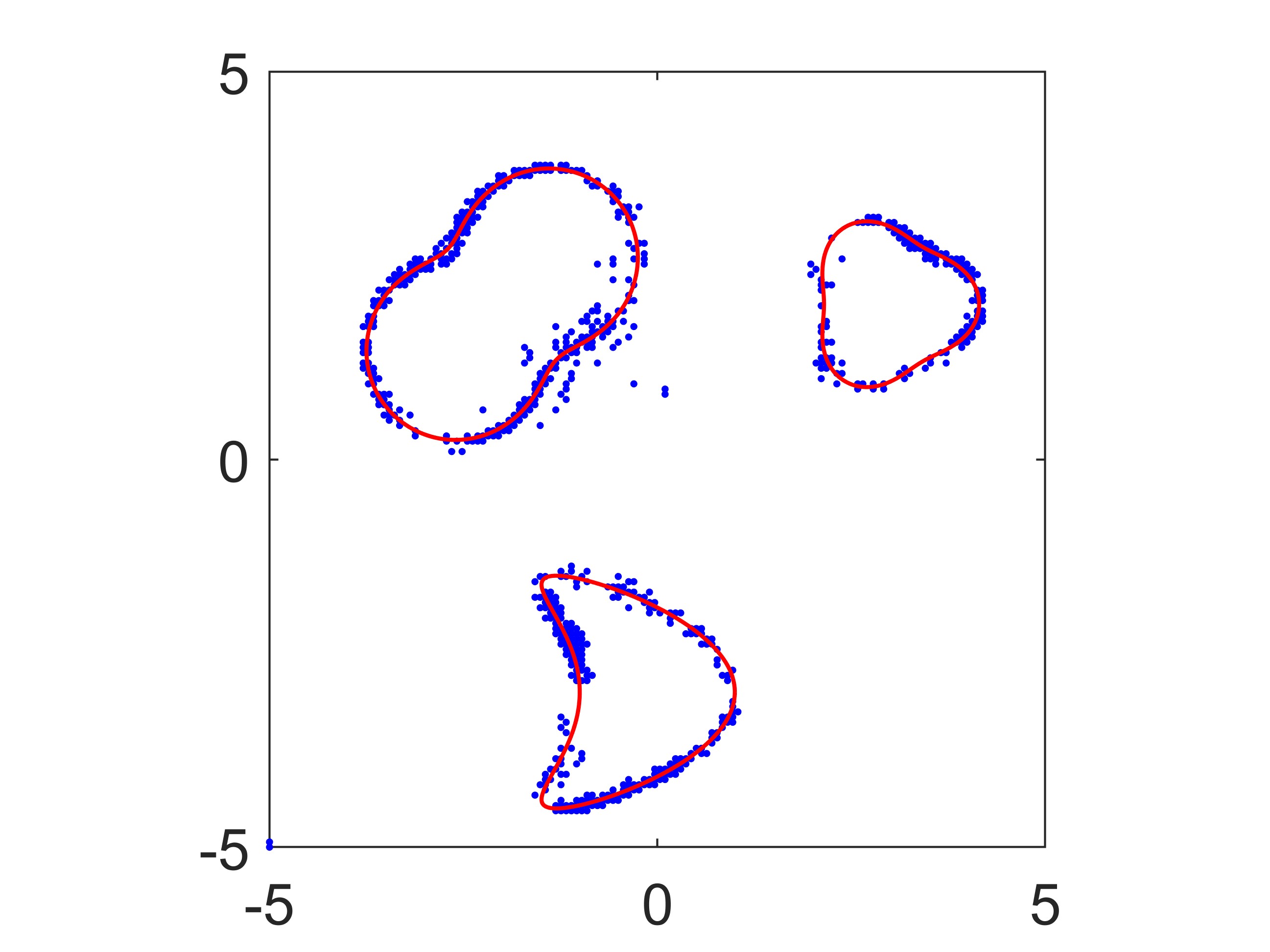}} 
  \subfigure[$g=0.05$]{\includegraphics[width=0.3\linewidth]{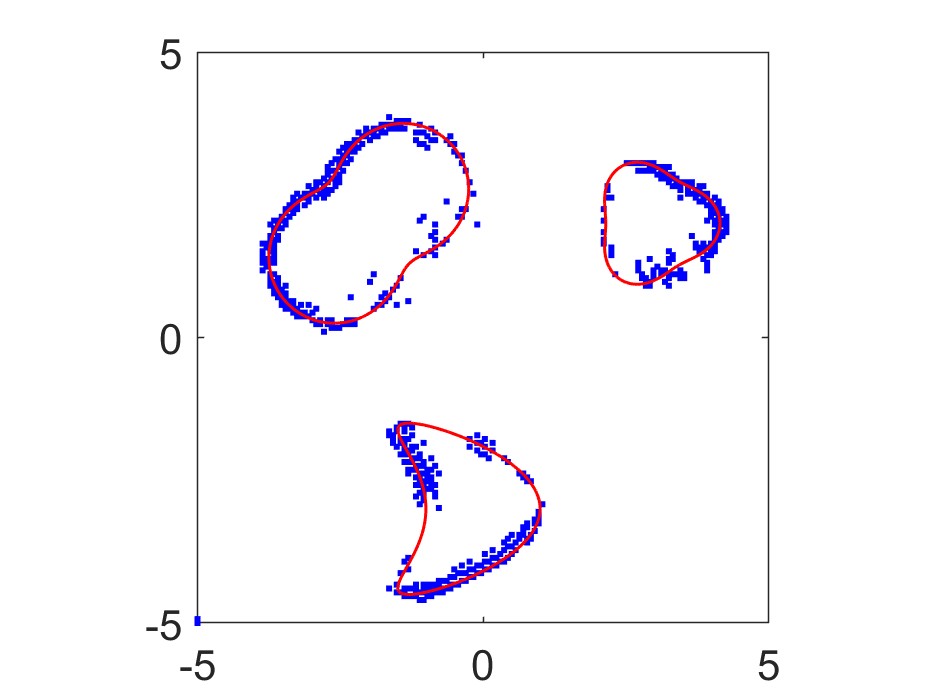}} 
  \caption{Recover the scatterer that has three disjoint components.}\label{fig: multi}
\end{figure}

\section{Conclusion}\label{sec: conclusion}

In this work, we propose an effective computational scheme for solving the inverse acoustic scattering problems. By using the tapered wave with a very narrow width illuminating the obstacle, the local boundary of the obstacle inside the tapered wave is reconstructed by a direct imaging algorithm. Then the shape of the target obstacle can be recovered by changing the incident direction of the tapered wave. Theoretical analysis is given to justify the rationale and simulation experiments are conducted to validate its applicability.

 Concerning future work, feasible extensions include the application of the imaging method to three-dimensional problems or more complicated physical configurations. Further extensions to the scenarios of electromagnetic or elastic waves are also interesting topics.


\end{document}